\documentclass[12pt,a4paper,psamsfonts]{amsart}
\usepackage{amssymb,amscd,amsxtra,calc}
\usepackage{cmmib57}

\theoremstyle{plain}
    \newtheorem{thm}{Theorem}[section]
    
    \newtheorem{claim}[thm]{Claim}
     \newtheorem{conjecture}[thm]{Conjecture}
    
    \newtheorem{example}[thm]{Example}
    \newtheorem{lemma}[thm]{Lemma}
    \newtheorem{proposition}[thm]{Proposition}
    
    \newtheorem{theorem}[thm]{Theorem}
    \newtheorem{mainTh}[thm]{Main Theorem}

\theoremstyle{definition}

    \newtheorem{remark}[thm]{Remark}
\theoremstyle{remark}

    \newtheorem{setup}[thm]{}

\newcommand{\BCC}{\mathbb{C}}
\newcommand{\C}{\mathbb{C}}

\newcommand{\PP}{\mathbb{P}}

\newcommand{\Q}{\mathbb{Q}}

\newcommand{\R}{\mathbb{R}}

\newcommand{\Z}{\mathbb{Z}}

\newcommand{\OO}{\mathcal{O}}

\newcommand{\Aut}{\operatorname{Aut}}

\newcommand{\Contr}{\operatorname{Contr}}
\newcommand{\Exc}{\operatorname{Exc}}
\newcommand{\Gal}{\operatorname{Gal}}
\newcommand{\GL}{\operatorname{GL}}

\newcommand{\id}{\operatorname{id}}

\newcommand{\NE}{\overline{\operatorname{NE}}}

\newcommand{\NS}{\operatorname{NS}}
\newcommand{\Null}{\operatorname{Null}}
\newcommand{\PE}{\operatorname{PE}}

\newcommand{\Proj}{\operatorname{Proj}}
\newcommand{\rank}{\operatorname{rank}}
\newcommand{\Sing}{\operatorname{Sing}}
\newcommand{\SL}{\operatorname{SL}}
\newcommand{\Supp}{\operatorname{Supp}}
\newcommand{\torsion}{\operatorname{torsion}}

\begin{document}

\title[$n$-dimensional projective varieties]
{$n$-dimensional projective varieties with the action of an abelian group of rank $n-1$}

\author{De-Qi Zhang}
\address
{
\textsc{Department of Mathematics} \endgraf
\textsc{National University of Singapore, 10 Lower Kent Ridge Road,
Singapore 119076
}}
\email{matzdq@nus.edu.sg}

\begin{abstract}
Let $X$ be a normal projective variety of dimension $n \ge 3$ admitting the action
of the group $G := \Z^{\oplus n-1}$ such that every non-trivial element of $G$ is of positive entropy.
We show:
`$X$ is not rationally connected' $\Rightarrow$ `$X$ is $G$-equivariant birational to the quotient of
a complex torus' $\Leftarrow \Rightarrow$ `$K_X + D$ is pseudo-effective for some
$G$-periodic effective fractional divisor $D$.' See Main Theorem \ref{CorA}.
To apply, one uses the above and fact:
`the Kodaira dimension $\kappa(X) \ge 0$' $\Rightarrow$ `$X$ is not uniruled'
$\Rightarrow$ `$X$ is not rationally connected.'
We may generalize the result to the case of solvable $G$ as in Remark \ref{rThB}.
\end{abstract}



\thanks{The author is supported by an ARF of NUS}

\maketitle

\section{Introduction}

We work over the field $\BCC$ of complex numbers.

For a normal projective variety $X$, we denote by $\NS(X)$ the Neron-Severi group, i.e., the
free abelian group of Cartier divisors modulo algebraic equivalence.
Let
$$\NS_{\R}(X) := \NS(X) \otimes_{\Z} \R .$$
It is a vector space over $\R$ of finite dimension (called the Picard number of $X$).
The {\it pseudo-effective divisor cone} $\PE(X)$ is the closure
in $\NS_{\R}(X)$ of effective divisor classes on $X$.
A divisor is {\it pseudo-effective} if its class belongs to $\PE(X)$.

A projective variety $X$ is {\it rationally connected}
if some (and hence every) resolution $X'$ of $X$ is rationally connected, i.e.,
any two points of $X'$ is connected by an irreducible rational curve on $X'$.
Rational varieties (and more generally unirational varieties) are rationally connected.
A variety of Kodaira dimension $\ge 0$ is not uniruled, so it is
not rationally connected (which is Condition (i) in Theorem \ref{ThA}).

An automorphism $g$ of a projective variety $X$ or its representation
${g^*}_{\, | \NS_{\R}(X)}$ is of {\it positive entropy}
if $g^*L = \lambda L$ for some nonzero nef $\R$-Cartier divisor $L$ and some $\lambda > 1$.
This is equivalent to saying that the action of $g$ on the total cohomology group $H^*(X', \R)$
of some (or equivalently every) $g$-equivariant resolution
$X'$ of $X$ has spectral radius $\rho(g) > 1$.
Indeed, $g^* |_{\NS_{\R}(X')}$ and $g^* |_{\NS_{\R}(X)}$ have the
same spectral radius since $g$ permutes the exceptional divisors of $X' \to X$.
Further, for smooth $X$, we may define the (topological) {\it entropy} as follows 
(see \cite[Proposition 5.8]{Di05} or \cite[Proposition 3.5]{Di12}, and also \cite[\S 2]{DS04}):
$$h(g) := \log (\rho(g)) .$$

Let $X$ be a normal projective variety of dimension $n$. 
In \cite{DS04}, Dinh and Sibony have proved that every commutative subgroup $G$ of $\Aut(X)$ of positive entropy
has rank $\le n-1$
(see also \cite{Z-Tits} for the extension to the solvable groups).
In Remarque 4.10 of their arXiv version, they also
mentioned the interest to study these $X$ equipped with some $G$ attaining the maximal rank $n-1$.

In this paper, we consider the maximal rank case. Precisely, we consider the hypothesis:

\begin{itemize}
\item[{\bf Hyp(sA)}]
{\it
$\Aut(X) \supseteq G := \Z^{\oplus n-1}$; every non-trivial element of $G$ is of positive entropy.
}
\end{itemize}

Theorem \ref{CorA} is our main result. For smooth varieties,
we have Theorem \ref{ThA}; see also key Theorem \ref{ThB} for the general case.
The most important assumption in Condition (ii) below is the pseudo-effectivity of the adjoint divisor $K_X + \delta \Delta$.
A Zariski closed subset $Z \subset X$ is $G$-{\it periodic} if $Z$ is stabilized (as a set)
by a finite-index subgroup of $G$.

Condition (ii) in Theorem \ref{ThA} below is natural and in fact necessary in order for
$X$ to be $G$-equivariant birational to a torus quotient.
See Main Theorem \ref{CorA}.
Precisely, Condition (i), together with Hyp(sA), imply that $X$ is non-uniruled and hence
$K_X$ is pseudo-effective, so Condition (ii) holds. See the proof of key Theorem \ref{ThB}.

\begin{theorem}\label{ThA}
Assume that $X$ and $G := \Z^{\oplus n-1}$ satisfy Hyp(sA) and $n = \dim X \ge 3$.
Furthermore, assume either one of the following three conditions.
\begin{itemize}
\item[(i)]
$X$ is {\rm not} rationally connected.
\item[(ii)]
$X$ is smooth. There exists a reduced simple normal crossing divisor $\Delta$
which is $G$-periodic such that
$K_X + \delta \Delta$ is a pseudo-effective divisor for some $\delta \in [0, 1)$;
here $K_X$ is the canonical divisor of $X$.
\item[(iii)]
$X$ is smooth. Every $G$-periodic proper subvariety of $X$ is a point.
\end{itemize}
Then, replacing $G$ by a finite-index subgroup, we have the following:
\begin{itemize}
\item[(1)]
There is a birational map $X \dashrightarrow Y$ such that
the induced action of $G$ on $Y$ is biregular and
$Y = T/F$, where $T$ is an abelian variety and
$F$ is a finite group whose action on $T$ is free outside
a finite subset of $T$.
\item[(2)]
The canonical divisor of $Y$ is torsion: $K_Y \sim_{\Q} 0$. To be precise,
$mK_Y \sim 0$ (linear equivalence) where $m = |F|$.
\item[(3)]
There is a faithful action of $G$ on $T$ such that
the quotient map $T \to T/F = Y$ is $G$-equivariant.
Every $G$-periodic proper subvariety of
$Y$ or $T$ is a point.
\end{itemize}
\end{theorem}

\begin{remark}\label{rThA}
(1) Under Condition (iii) of Theorem \ref{ThA} (or key Theorem \ref{ThB}),
we can take $Y = X$ (cf.~Lemma \ref{ample}). See also Remark \ref{rHyp} (1).

(2) Hyp(sA) is birational in nature (cf.~Lemma \ref{birAct}).
Conditions (i) and (ii) are also birational in nature. (i) is clear by the definition.
For (ii), see the proof in Remark \ref{rHyp}.

(3) The condition $n = \dim X \ge 3$ is needed to kill the second Chern
class $c_2$ of $Y$ by using the perpendicularity of $c_2$ with the product
of common nef eigenvectors of $G$.
\end{remark}

In summary, Theorem \ref{ThA} (or key Theorem \ref{ThB}) and Proposition \ref{PropC} say that
if $G := \Z^{\oplus n-1}$ ($n \ge 3$) acts on a projective $n$-fold $X$ and
every non-trivial element of $G$ is of positive entropy then:
$$X \, \text{\it is not rationally connected} \, \Rightarrow \,
X \, \overset{\rm birational}\sim \, \text{\it torus quotient;}$$
$$K_X + D \ \text{\it is pseudo-effective for a G-periodic fractional} \, D
\Leftarrow \Rightarrow X \overset{\rm birational}\sim \text{\it torus quot.} $$
For applications, one may combine the above with the following well known fact,
where $\kappa(X)$ is the Kodaira dimension of $X$:
$$
\kappa(X) \ge 0 \, \Rightarrow \, X \, \text{\it is not uniruled} \, \Rightarrow \, X \,
\text{\it is not rationally connected} .
$$

Setting $\delta = 1$, the limit case in Condition (ii) of Theorem \ref{ThA},
we have the following Proposition \ref{PropD} which is a special case of Proposition \ref{PropE}.

\begin{proposition}\label{PropD}
Assume that $X$ and $G := \Z^{\oplus n-1}$ satisfy Hyp(sA) and $n = \dim X \ge 2$.
Assume further that $X$ is smooth, $\Delta$ is a simple normal crossing reduced divisor
which is $G$-periodic, and
$K_X + \Delta$ is a pseudo-effective divisor.
Then there is a birational map $X \dashrightarrow Y$
such that the following are true.
\begin{itemize}
\item[(1)]
$Y$ is a normal projective variety. The map $X \dashrightarrow Y$
is surjective in codimension-1.
The induced action of $G$ on $Y$ is biregular.
\item[(2)]
We have $K_Y + \Delta_Y \sim_{\Q} 0$, where
$\Delta_Y$ is the direct image of $\Delta$ and a reduced divisor.
\item[(3)]
Every $G$-periodic positive-dimensional proper
subvariety of $Y$ is contained in $\Delta_Y$.
\end{itemize}
\end{proposition}

Without the pseudo-effectivity of $K_X + \Delta$ in Proposition \ref{PropD},
the following conjecture is well known.
The implication `$(2) \Rightarrow (1)$' is known in any dimension,
while the implication `$(1) \Rightarrow (2)$' is known only in dimension $\le 2$.
See \cite[Lemma 5.11, Theorem 1.1]{KeMc}.

\begin{conjecture}\label{conj}
Let $X$ be a smooth projective variety and let $\Delta$
be a simple normal crossing reduced divisor.
Then the following are equivalent.
\begin{itemize}
\item[(1)]
$K_X + \Delta$ is not a pseudo-effective divisor.
\item[(2)]
$X \setminus \Delta$ is covered by images of the affine line $\C$.
\end{itemize}
\end{conjecture}

\begin{setup} {\bf Applications of our results.}
{\rm
Suppose that $X$ and $G$ satisfy Hyp(sA). If every $G$-periodic proper subvariety of $X$
is a point then $X = Y$ is a torus quotient (cf.~Theorem \ref{ThA} and Lemma \ref{ample}).
Otherwise, since the union of
all $G$-periodic positive-dimensional proper subvarieties of $X$ is a
Zariski-closed proper subset of $X$ (cf.~Lemmas \ref{ELMNP} and \ref{null}),
replacing $X$ by a $G$-equivariant blowup, we may assume that
$X$ is smooth and this union is a divisor, denoted as $\Delta$,
with only simple normal crossings.

If $K_X + \delta \Delta$ is pseudo-effective for some $\delta \in [0, 1)$
(resp. for $\delta = 1$), we may apply Theorem \ref{ThA} (resp. Proposition \ref{PropD})
and say that $X$ is $G$-equivariant birational to a torus quotient
(resp. a variety $Y$ with $K_Y + \Delta_Y \sim_{\Q} 0$).

If $K_X + \Delta$ is not pseudo-effective, then Conjecture \ref{conj} (confirmed when $\dim X \le 2$)
asserts that $X \setminus \Delta$ is covered by images of the affine line $\C$,
i.e., $X$ is covered by projective rational curves $\{C_t\}$
with $C_t$ meeting the $G$-periodic locus $\Delta$ at most once.
}
\end{setup}

\begin{example}
{\rm
In any dimension, there are examples satisfying Hyp(sA) and Condition (ii) of Theorem \ref{ThA}.
Indeed, denote by
$E$ the elliptic curve $\C/(\Z + \sqrt{-1} \Z)$, and
by $T := E \times \cdots \times E$ ($n$ copies of $E$), an abelian variety of dimension $n$.
Then $\Aut(T)$ contains $G := \Z^{\oplus n-1} \subset \GL_n(\Z)$
and $G$ has a natural action on $T$,
such that $T$ and $G$ satisfy Hyp(sA) (cf.~\cite[Example 4.5]{DS04}).
Consider the automorphism:
$$f : T \to T, \,\, (x_1, \dots, x_n) \, \mapsto \, (\sqrt{-1}x_1, \dots, \sqrt{-1} x_n) .$$
Denote by $X$ the quotient variety  $T/\langle f \rangle$.
Then $K_X \sim_{\Q} 0$ and $X$ has only $\Q$-factorial {\it Kawamata log terminal (klt)
singularities}. Further, when $n \ge 4$, $X$ is a {\it Calabi-Yau variety} in the sense that
$X$ has only {\it canonical singularities}, $K_X \sim_{\Q} 0$, and the irregularity $q(X) = 0$.
When $n \le 3$, $X$ is a rationally connected variety, in fact, a rational variety,
see \cite[Theorem 5.9]{Og}.

Since the diagonal group $\langle f \rangle$ is normalized by $G$,
the action of $G$ on $T$ descends to a faithful action on $X$.
This $X$ and $G$ satisfy Hyp(sA), and Hyp(B) in \S \ref{main} with $D = 0$.
Equivalently, a $G$-equivariant resolution $X'$ of $X$,
and $G$ satisfy Hyp(sA) and Condition (ii) in Theorem \ref{ThA} (cf.~Proposition \ref{PropC}).
}
\end{example}

\begin{remark}\label{rdim2}
In dimension two, there are examples satisfying Hyp(sA),
and all the other conditions of Proposition \ref{PropD} (resp.
the equivalent conditions of Conjecture \ref{conj}).
{\it However, we do not know whether there are such examples in dimension $\ge 3$}.

Indeed, there are surfaces $X_i$, for $i = 1$ (resp. $i = 2$) and an automorphism $g_i$ of
positive entropy such that the union $\Delta_i$ of {\it all} $g_i$-periodic curves
is a simple normal crossing reduced divisor and
$K_{X_i} + \Delta_i$ is a pseudo-effective (resp. non-pseudo-effective) divisor.

For $i = 1$, take $X_1'$ to be the $12$-point blowup of $\PP^2$ as constructed in \cite[Example 3.3]{Di}
with a smooth elliptic curve $C_1$ being $g_1$-periodic, 
and $X_1 \to X_1'$ a $g_1$-equivariant blowup such that
$\Delta_1$ is a simple normal crossing divisor. Since $\Delta_1$ contains the elliptic curve $C_1'$
lying over $C_1$,
we have $K_{X_1} + \Delta_1 \ge K_{X_1} + C' \ge 0$ by the Riemann-Roch theorem.

For $i = 2$, take $X_2'$ to be the $10$-point blowup of $\PP^2$
as in \cite{BK} or \cite[Theorem 1.1]{Mc} such that $-K_{X_2'} \sim C_2$ for a cuspidal
rational curve $C_2$. Take the $g_2$-equivariant
blowup $X_2 \to X_2'$ of the cusp $P$ of $C_2$ and two infinitely near points of $P$
such that the inverse $\Delta_2$ of $C_2$ is a simple normal crossing (reduced) divisor.
Then the adjoint divisor $K_{X_2} + \Delta_2$ is not pseudo-effective.
Precisely, the adjoint divisor is linearly equivalent to $-E$ with $E$ the last $(-1)$-curve
in the blowup $X_2 \to X_2'$.
\end{remark}

\begin{setup}
{\bf Related works}.
When $G$ is an abelian subgroup of $\Aut(X)$, Dinh-Sibony \cite[Theorem 1]{DS04}
proved that the rank of the part of $G$ with positive entropy is $\le \dim X - 1$.
Hence Theorem \ref{ThA} deals with exactly the maximal rank case.

Partial results pertaining to Theorem \ref{ThA} have been obtained in
\cite{max} and \cite{NullG}, where one imposed either the assumption that $X$ and the pair $(X, G)$
are minimal, or the assumption that $X$ has no $G$-periodic positive-dimensional proper subvarieties.
These assumptions
seem a bit strong, because $X$ may not have a minimal model $X'$ when $X$ is uniruled
and the regular action of $G$ on $X$ usually induces only a
{\it birational} action on $X'$ (if exists).

In the key Theorem \ref{ThB} (singular version), 
we allow $X$ to have Kawamata log terminal (klt) singularities,
which is more natural, from the viewpoint of Log Minimal Model Program (LMMP),
than quotient singularities in the previous papers \cite{max} and \cite{NullG}.

See Remark \ref{rThB} for the generalization of Theorem \ref{ThA} (smooth version)
and key Theorem \ref{ThB} to the case with $G$ solvable, like some
related partial results in the paper \cite{NullG}.

Our Theorem \ref{ThA}
can be compared with Katok-Hertz \cite[Corollary 7]{KH},
where they consider the smooth action of $\Z^{\oplus n-1}$
on a smooth (real) manifold $X'$ of dimension $n$ with the conclusion:
$X'$ is homeomorphic to the connected sum of the compact $n$-torus with another manifold.

Theorem \ref{ThA} can also be compared with Cantat - Zeghib \cite{CZ} where the authors
impose the assumption that $X$ admits the action of a lattice $\Gamma$ of rank $n-1$ in an
{\it almost simple} real Lie group $H$
(which turns out to be isogenous to $\SL_n(\R)$ or $\SL_n(\C)$)
and deduce a similar conclusion. They use Margulis' super-rigidity for lattices of higher rank,
work out in detail all the actions of $\Gamma$ on complex $n$-tori
and also consider the rank $n+1$ case (necessarily of non-positive entropy).

Our assumption in Theorem \ref{ThA} is weaker since $\Z^{\oplus n-1}$
can be thought of as a small part of
the lattice $\SL_n(\Z)$ in $H$, and hence no super-rigidity is available to induce the action of
a big Lie group $H$. Instead, we fully utilize the entropy-positivity of $G$.
Hence our approach is more dynamic in spirit, combined with some algebro-geometric tools.
\end{setup}

\begin{setup}
{\bf The three ingredients for the proof of Theorems \ref{ThA} and \ref{ThB}}

(1) Given a pair $(X, D)$ of a mildly singular variety $X$ and a boundary divisor
$D$ supported on $G$-periodic divisors,
we run formally the log minimal model program (LMMP) and reach a new pair $(X', D')$
after a divisorial contraction or a flip. See \cite[\S 3.7]{KM}.
We are able to handle the case when the new pair admits a Fano fibration.
The unknown termination of sequence of flips is avoided
by running the LMMP directed by
an ample divisor.
See \cite{BCHM} or \cite{Bi}.
A difficulty occurs: our $G$ is
an infinite group and $X'$ may have infinitely many extremal rays;
how to descend the action of $G$ on $X$ to a {\it biregular} action of $G$ on $X'$?
To overcome this, we have managed to {\it algebraically} contract only $G$-periodic subvarieties,
i.e., $G$-periodic extremal rays, and make the LMMP $G$-equivariant.

(2) Yau's deep result characterizing \'etale quotient (or $Q$-torus)
in terms of the vanishing of the first and second Chern classes has been extended to
mildly singular cases, thanks to the recent work of \cite{GKP}.

(3) We use a type of Zariski-decomposition of the adjoint divisor $K_X + D$ as in \cite{ZDA}.
\end{setup}

To make this paper easily accessible, we use
only the formal process, also sketched in the paper,
of the log minimal model program (LMMP), with
no detailed technicality involved or required.

The result of this paper explains why manifolds $Y$ found so far
(like \cite[Example 4.5]{DS04})
with maximal number of commutative automorphisms of positive entropy,
{\it all have trivial first Chern class} $c_1(Y)$, or $K_Y \sim_{\Q} 0$, after birational change of models.

\par \vskip 1pc
{\bf Acknowledgement.}
I would like to thank N. Nakayama for coming up with a clean proof of Lemma \ref{HR},
after a few messy ones of mine, and the referee for the suggestions.

\section{More general results for singular varieties}\label{main}

We refer to \cite{KM} for the conventions and the definitions of {\it Kodaira dimension}, and
{\it Kawamata log terminal (klt)}, {\it divisorial log terminal (dlt)},
{\it canonical}, or {\it log canonical singularities};
see \cite[Definition 2.34, 7.73]{KM}.

We consider the following hypotheses for normal projective varieties $X$ and $W$
and a group $G$ of automorphisms.
Denote by $K_X$ and $K_W$ their {\it canonical divisors}.

\begin{itemize}
\item[{\bf Hyp(A)}]
$G \le \Aut(X)$ is a subgroup.
The representation $G^* := G_{|\NS_{\R}(X)}$ is isomorphic to
$\Z^{\oplus n-1}$ where $n = \dim X$.
Every element of $G^* \setminus \{\id\}$
is of positive entropy.

\item[{\bf Hyp(B')}]
$W$ has Kodaira dimension $\kappa(W) \ge 0$.

\item[{\bf Hyp(B'')}]
$W$ is non-uniruled, i.e., $W$ is not covered by rational curves.

\item[{\bf Hyp(B''')}]
$G \le \Aut(W)$ is a subgroup. For some $G$-equivariant resolution
$\eta : X \to W$ such that the inverse $\Delta := \eta^{-1}(\Sing W)$
of the singular locus $\Sing W$ of $W$
is a simple normal crossing divisor, $K_X + \delta \Delta$
is a pseudo-effective divisor for some $\delta \in [0, 1)$.

\item[{\bf Hyp(B)}]
$G \le \Aut(X)$ is a subgroup. For some effective $\R$-divisor $D$
whose irreducible components are $G$-periodic,
the pair $(X, D)$ has at worst $\Q$-factorial klt singularities, and
$K_X + D$ is a pseudo-effective divisor.
\end{itemize}

\begin{remark}\label{rHyp}
(1) If $X$ and $G$ satisfy Hyp(A),
then the union of all {\it positive}-dimensional $G$-periodic proper subvarieties of $X$
is a Zariski closed proper subset of $X$. See Lemmas \ref{ELMNP} and \ref{null}.

(2) Hyp(sA) in the Introduction is stronger than Hyp(A).
Hyp(A) is a birational property or more generally a property
preserved by generically finite maps.
See Lemma \ref{birAct}.

(3) Hyp(B') and Hyp(B'') are birational conditions in nature, by the definition.

Hyp(B) is also a birational condition in nature. Indeed, suppose that
$\sigma : X' \to X$ is a $G$-equivariant birational morphism with $X'$ being $\Q$-factorial.
We can write
$$K_{X'} + \sigma'D + E(1) = \sigma^*(K_X + D) + E(2)$$
where $\sigma'D$ is the proper transform of $D$,
where $E(1)$ and $E(2)$ are $\sigma$-exceptional (and hence $G$-periodic) effective divisors
with no common components.
Now $D' := \sigma'D + E(1)$ has components all $G$-periodic.
The above display shows that the pair $(X', D')$ is klt since so is $(X, D)$. Further,
$K_{X'} + D'$ is pseudo-effective if $K_X + D$ is pseudo-effective as in Hyp(B);
the converse is also true by taking the direct image $\sigma_*$ of the display.

Similarly, Hyp(B''') is a birational condition in nature, by proving as above or
using the logarithmic ramification divisor formula.
\end{remark}

\begin{proposition}\label{PropB}
\begin{itemize}
\item[(1)]
Hyp(B') implies Hyp(B'').
\item[(2)]
Hyp(B'') implies Hyp(B''') for any subgroup $G \le \Aut(W)$.
\item[(3)]
If $W$ and $G$ satisfy Hyp(B''') then the $G$-equivariant blowup $X$ in Hyp(B'''),
and $G$ satisfy Hyp(B).
\end{itemize}
\end{proposition}

\begin{proof}
For (1), Hyp(B') says that some (or equivalently every) resolution $W'$ of $W$
has Kodaira dimension $\kappa(W') \ge 0$. Hence $mK_{W'} \sim B$
for some integer $m > 0$ and effective divisor $B$. Thus $K_{W'}$ is pseudo-effective.
This means $W'$ (or equivalently) $W$ is not uniruled, by the uniruled criterion
of Miyaoka-Mori and Boucksom-Demailly-Paun-Peternell.

For (2), if $W$ is non-uniruled then so is its blowup $X$. Thus $K_X$ is pseudo-effective.
So Hyp(B''') holds with $\delta = 0$.

For (3) and the $W$ and $X$ in Hyp(B'''), $\Sing W$ and hence its inverse
$\Delta$ on $X$ are stabilized by $G$. Set $D := \delta \Delta$.
Then Hyp(B) is satisfied by $X$ and $G$.
\end{proof}

Proposition \ref{PropC} below, together with \ref{PropB} above,
say that under Hyp(A), the Hyp(B) in Condition (ii) of Theorem \ref{ThB} or \ref{ThA}
is not just sufficient but also necessary for $X$ to be
$G$-equivariant birational to a torus quotient. See Main Theorem \ref{CorA}.

\begin{proposition}\label{PropC}
Let $\sigma: T \to W$ be a finite morphism from an abelian variety $T$
onto a normal projective variety $W$
which is \'etale in codimension-1 and equivariant under the action of some group $G$.
Then $W$ and $G$ satisfy Hyp(B''').
\end{proposition}

\begin{proof}
Since $\sigma$ is \'etale in codimension-1, $K_T = \sigma^*K_W$.
Since $K_T \sim 0$, we have $K_W \sim_{\Q} 0$, i.e., $mK_W \sim 0$ with $m = \deg \sigma$.
Since $T$ is smooth and hence klt, so is $W$; see \cite[Proposition 5.20]{KM}.
Let 
$$\eta : X \to W$$ 
be a $G$-equivariant resolution such that the inverse $\Delta = \eta^{-1}(\Sing W)$
of the singular locus $\Sing W$ of $W$
is a simple normal crossing (reduced) divisor.
Write
$$K_X + E(1) = \eta^*K_W + E(2)$$
where $E(1)$ and $E(2)$ are $\eta$-exceptional
effective divisors with no common components.
Since $W$ is klt, $E(1) = \sum e_i E_i$ is fractional, i.e., $e_i \in (0, 1)$.
Choose $\delta \in (0, 1)$ such that $E(1) \le \delta \Delta$.
Then 
$$K_X + \delta \Delta \ge K_X + E(1) = \eta^*K_W + E(2) \sim_{\Q} E(2)$$
hence $K_X + \delta \Delta$ is a pseudo-effective divisor.
Therefore, $W$ and $G$ satisfy Hyp(B''').
\end{proof}

Theorem \ref{ThB} below is the key step in proving main Theorem \ref{CorA}.

Regarding the conditions about singularities in (i) and (ii) below,
if $X$ is smooth and $\Delta$ is a reduced divisor with only simple normal crossings then
both $X$ and the pair $(X, \delta \Delta)$, with $\delta \in [0, 1)$,
automatically have at worst $\Q$-factorial klt singularities.

\begin{theorem}\label{ThB}
Assume that $X$ and $G$ satisfy Hyp(A) and $n = \dim X \ge 3$.
Furthermore, assume either one of the following three conditions.
\begin{itemize}
\item[(i)]
$X$ is not rationally connected.
$X$ has only $\Q$-factorial klt singularities.
\item[(ii)]
$X$ and $G$ satisfy Hyp(B).
\item[(iii)]
$X$ has only  klt singularities.
Every $G$-periodic proper subvariety of $X$ is a point.
\end{itemize}

Then, replacing $G$ by a finite-index subgroup, we have the following:
\begin{itemize}
\item[(1)]
There is a birational map $X \dashrightarrow Y$ such that
the induced action of $G$ on $Y$ is biregular and
$Y = T/F$, where $T$ is an abelian variety and
$F = \Gal(T/Y)$ acts on $T$ freely outside
a finite subset of $T$.
\item[(2)]
The canonical divisor $K_Y$ satisfies
$mK_Y \sim 0$ (linear equivalence), where $m = |F|$.
\item[(3)]
The action of $G$ on $Y$ lifts to an action of a group
$\widetilde{G}$ on $T$ with $\widetilde{G}/\Gal(T/Y) \cong G$.
Every $G$- (resp. $\widetilde{G}$-) periodic proper subvariety of
$Y$ (resp. $T$) is a point.
\newline
{\rm More precisely, we have:}

\item[(4)]
There is a sequence $\tau_s \circ \cdots \circ \tau_0$ of birational maps:
$$X = X(0) \overset{\tau_0}\dashrightarrow X(1) \overset{\tau_1}\dashrightarrow \cdots
\dashrightarrow X(s) \overset{\tau_s}\to X(s+1) = Y$$
such that
each $X(j) \dasharrow X(j+1)$
($0 \le j < s$) is either a birational morphism
or an isomorphism in codimension-$1$.
$\tau_s$ is a birational morphism.
\item[(5)]
Each pair $(X(i), D(i))$ ($0 \le i \le s+1$) has at worst klt singularities,
where $D(i) \subset X(i)$ is the
direct image of $D$. The $X(j)$ ($0 \le j \le s$) is $\Q$-factorial.
\item[(6)]
The induced action of $G$ on each $X(i)$ ($0 \le i \le s+1$) is biregular,
\item[(7)]
$K_{X(s)} = \tau_s^*K_Y \sim_{\Q} 0$ (i.e., $m' K_{X(s)} \sim 0$
for some $m' > 0$), and $D(s) = 0$.
\end{itemize}
\end{theorem}

Hyp(B) in Theorem \ref{ThB} is actually necessary in order for
$X$ to be $G$-equivariant birational to a torus quotient:

\begin{mainTh}\label{CorA}
Assume that $X$ and $G$ satisfy Hyp(A) and $n = \dim X \ge 3$.
Then the following are equivalent.
\begin{itemize}
\item[(1)]
Replacing $G$ by a finite-index subgroup and $X$ by a $G$-equivariant birational model,
(the new) $X$ and $G$ satisfy Hyp(B).
\item[(2)]
Replacing $G$ by a finite-index subgroup, there exist a $G$-equivariant birational
map $X \dasharrow W$ and a finite morphism $T \to W$ from an abelian variety $T$
which is Galois and \'etale in codimension-1, such that
the action of $G$ on $W$ lifts to an action of a group
$\widetilde{G}$ on $T$ with $\widetilde{G}/\Gal(T/W) \cong G$.
\par \noindent
(We can take $\widetilde{G} = G$ if Hyp(sA) is satisfied).
\end{itemize}
\end{mainTh}

If we weaken the klt condition for the pair $(X, D)$ in
Hyp(B) for Theorem \ref{ThB} to being dlt, we have the following
result which contains Proposition \ref{PropD} as a special case
since the pair $(X, \Delta)$ there is automatically $\Q$-factorial dlt.

\begin{proposition}\label{PropE}
Assume that $X$ and $G$ satisfy Hyp(A) and $n = \dim X \ge 2$.
Assume further that for some effective $\R$-divisor $D$
whose irreducible components are $G$-periodic,
the pair $(X, D)$ has at worst $\Q$-factorial divisorial log terminal (dlt) singularities, and
$K_X + D$ is a pseudo-effective divisor.
Then there is a birational map $X \dashrightarrow Y$
such that:
\begin{itemize}
\item[(1)]
$Y$ is a normal projective variety. The map $X \dashrightarrow Y$ is surjective in codimension-1.
The induced action of $G$ on $Y$ is biregular.
\item[(2)]
The pair $(Y, D_Y)$ has only log canonical singularities
and $K_Y + D_Y \sim_{\Q} 0$, where
$D_Y$ is the direct image of $D$.
\item[(3)]
Every $G$-periodic positive-dimensional proper
subvariety of $Y$ is contained in the support of $D_Y$.
\end{itemize}
\end{proposition}

We may generalize Theorem \ref{ThB} to the case of solvable $G$:

\begin{remark}\label{rThB}
Assume that $X$ is a smooth projective variety of dimension $n \ge 3$. Let
$G \le \Aut(X)$ and let $N(G) \subseteq G$ be the set of elements $g \in G$
of null entropy : $h(g) = 0$, i.e., every eigenvalue of $g^* \, |_{\NS_{\C}(X)}$
has modulus $1$.
Assume that the image $G \to \GL(\NS_{\C}(X))$ is solvable
and has connected Zariski-closure, and
that $G/N(G) \cong \Z^{\oplus n-1}$.
This assumption is weaker than Hyp(A). By \cite[Theorem 2.2]{NullG},
with $H^{1,1}(X)$ in its proof replaced by $\NS_{\C}(X)$ and $G$ replaced by a finite-index subgroup,
we have $G = N(G) \, H$
where $H$ is a subgroup of $G$ such that $H \, |_{\NS_{\C}(X)} \cong \Z^{\oplus n-1}$
and $N(G) \, |_{\NS_{\C}(X)}$ is unipotent.
Now we can apply Theorem \ref{ThB} to our $X$ and $H$ here.
Especially, if $X$ is not rationally connected, then $X$ is
$H$-equivariant birational to a torus quotient.
\end{remark}

\begin{setup}
{\bf Theorem \ref{ThB} implies Theorem \ref{ThA}}
{\rm
\par
Note that Hyp(sA) implies Hyp(A).
Condition (iii) in Theorem \ref{ThA} clearly implies Condition (iii) in Theorem \ref{ThB}.
Condition (ii) in Theorem \ref{ThA} implies Hyp(B) with $D := \delta \Delta$.
Thus we can apply Theorem \ref{ThB} with Condition (ii).
Assume Condition (i) in Theorem \ref{ThA}. Replacing $X$ by a $G$-equivariant
resolution, the (new) $X$ is still
not rationally connected and satisfies Hyp(A), after replacing $G$ by a finite-index subgroup;
see Lemma \ref{birAct}.
Thus we can apply Theorem \ref{ThB} with Condition (i).

Now under Condition (i), (ii) or (iii), Theorem \ref{ThA} follows from Theorem \ref{ThB}.
Indeed, when $G = \Z^{\oplus n-1}$,
replacing $G$ by a finite-index subgroup, the faithful action of $G$ on $Y$ lifts to
a faithful action of $G$ on $T$ (cf.~\cite[Lemma 2.4, \S 2.15]{max}).
This proves Theorem \ref{ThA}.
}
\end{setup}

\section{Preliminary results}

\begin{lemma}\label{birAct}
Assume that a group $G$ acts biregularly on normal projective varieties $X_1$ and $X_2$
of the same dimension $n$.
Let $\sigma : X_1 \dashrightarrow X_2$
be a $G$-equivariant generically finite map.
\begin{itemize}
\item[(1)]
Suppose that $\sigma$ is a {\rm morphism}, or
$\sigma$ is an isomorphism in codimension-1
with both $X_i$ being $\Q$-factorial.
If $X_1$ and $G$ satisfy Hyp(A) then so do $X_2$ and $G$.
\item[(2)]
Suppose that $X_a$ and $G$, for some $a$ in $\{1, 2\}$, satisfy Hyp(A).
Then, replacing $G$ by a finite-index subgroup, $X_i$ and $G$ satisfy Hyp(A) for both $i$ in $\{1, 2\}$.
\end{itemize}
\end{lemma}

\begin{proof}
Suppose that $\sigma$ is a morphism.
Identify the representation $G$ on $\NS_{\R}(X_2)$ with that on the subspace
$\sigma^*\NS_{\R}(X_2) \subseteq \NS_{\R}(X_1)$.
Let $K \le G$ be the subgroup such that the following natural
sequence of homomorphisms is exact
$$1 \to K_{| \, \NS_{\R}(X_1)} \to G_{| \, \NS_{\R}(X_1)} \overset{r}\to G_{| \, \sigma^*\NS_{\R}(X_2)} \to 1$$
where $r$ is the restriction homomorphism and necessarily surjective.
If $g$ is in $K$ then $g^*$ fixes the class $[H'] = [\sigma^*H]$ with $H$ any ample divisor on $X_2$;
here $H' = \sigma^*H$ is a nef and big divisor on $X_1$, since $\sigma$ is generically finite.
Thus
$$K \, \le \, \Aut_{[H']}(X_1) := \{g \in \Aut(X_1) \, | \, g^*[H'] = [H']\}$$
where the latter group is a finite extension of the identity connected component $\Aut_0(X_1)$
of $\Aut(X_1)$,
by Lieberman \cite[Proposition 2.2]{Li} or \cite[Lemma 2.23]{JDG}.
Since the continuous group $\Aut_0(X_1)$ acts on the integral lattice
$\NS(X_1)/(\torsion)$ as identity,
the representation $K_{| \, \NS_{\R}(X_1)}$ is a finite subgroup of
$G_{| \, \NS_{\R}(X_1)}$.

If $X_1$ and $G$ satisfy Hyp(A), then $G_{| \, \NS_{\R}(X_1)}$ is isomorphic
to $\Z^{\oplus n-1}$ and its only finite subgroup is $\{\id\}$.
Thus $K_{| \, \NS_{\R}(X_1)} = \{\id\}$.
So we have the isomorphisms
$$\Z^{\oplus n-1} \cong G_{| \, \NS_{\R}(X_1)} \overset{r}\cong G_{| \, \sigma^*\NS_{\R}(X_2)}
\cong G_{| \, \NS_{\R}(X_2)}.$$
Hence $G$ and $X_2$ satisfy Hyp(A) and we have proved
(1) for the first situation,
noting that $g \in G$ acts on $X_1$ with positive entropy if and only if the same holds on $X_2$.
See \cite[Lemma 2.6]{JDG}.
In the second situation of (1), we can identify $\NS_{\R}(X_1)$ and $\NS_{\R}(X_2)$
so the result is clear.

For (2), let $W \subset X_1 \times X_2$ be the graph of $\sigma$, so $G$ acts on $W$ biregularly.
Now the two natural projections $p_i : W \to X_i$ are
both $G$-equivariant generically finite morphisms.
By (1), it suffices to prove the assertion that $W$ and $G$
satisfy Hyp(A), after replacing $G$ by a finite-index subgroup.
By the argument in (1), we have an exact sequence
$$1 \to K_{| \, \NS_{\R}(W)} \to G_{| \, \NS_{\R}(W)} \overset{r}\to G_{| \, p_a^*\NS_{\R}(X_a)} \to 1.$$
Thus we have isomorphisms
$$G_{| \, \NS_{\R}(W)}/K_{| \, \NS_{\R}(W)} \cong G_{| \, p_a^*\NS_{\R}(X_a)} \cong
G_{| \, \NS_{\R}(X_a)} \cong \Z^{\oplus n-1}.$$
This and the finiteness of $K_{| \, \NS_{\R}(W)}$ as shown in (1),
imply that $G_{| \, \NS_{\R}(W)} \cong \Z^{\oplus n-1}$, after replacing $G$
by a finite-index subgroup; see \cite[Lemma 2.4]{max}.
This proves the required assertion and also (2).
\end{proof}

The following clean proof of the Hodge index theorem for singular varieties is due to N. Nakayama.
Another proof in \cite[Lemma 2.5]{max} works only when $M$ is nef.

\begin{lemma}\label{HR}
Let $X$ be a projective variety of dimension $n \ge 2$.
Let $H_1, \cdots, H_{n-1}$ be ample $\R$-divisors
and $M$ an $\R$-Cartier divisor.
Suppose that
$H_1 \cdots H_{n-1} \cdot M = 0 = H_1 \cdots H_{n-2} \cdot M^2$.
Then $M \equiv 0$ (numerical equivalence).
\end{lemma}

\begin{proof}
Let $\Sigma := \{L \in \NS_{\R}(X) \, | \, H_1 \cdots H_{n-1} \cdot L = 0\}$.
Lemma \ref{HR} follows from Property (HI): higher-dimensional
Hodge index theorem holds for $\R$-Cartier divisors, i.e.,
the quadratic form $I(L_1, L_2) := (H_1 \cdots H_{n-2} \cdot L_1 \cdot L_2)$ is
negative definite on $\Sigma$. Indeed, Property (HI) holds when $X$ is smooth, or
when the $H_i$ are all $\Q$-divisors by cutting $X$ by general hypersurfaces
in multiples of $H_i$ and reducing to the surface case
so that the usual Hodge index theorem can be applied.

Now we consider the general case where the $H_i$ are $\R$-divisors.

\begin{claim}\label{sNeg}
Let $P_1, \dots, P_{n-1}$ be ample $\R$-divisors and
$N$ an $\R$-Cartier divisor such that
$P_{1} \cdots P_{n-1} \cdot N = 0$. Then $P_{1} \cdots P_{n - 2} \cdot N^2 \leq 0$.
\end{claim}

We prove Claim \ref{sNeg}.
Take ample $\Q$-divisors
$P_{i, m}$ such that $P_{i} = \lim_{m \to \infty} P_{i, m}$.
There is a unique real number $r(m)$ such that
$$P_{1, m} \cdots \cdot P_{n-1, m} \cdot (N + r(m)P_{n-1}) = 0$$
since $P_{1, m} \cdots P_{n-1, m} \cdot P_{n-1} > 0$.
By the assumption on $N$, $\lim_{m \to \infty} r(m) = 0$.
Since the Hodge index theorem holds for $\Q$-divisors $P_{i, m}$
as mentioned early on, we have
$$P_{1, m} \cdots P_{n-2, m}(N + r(m)P_{n-1})^{2} \leq 0 .$$
Letting $m \to \infty$, we have
$P_{1} \cdots P_{n-2} \cdot N^{2} \leq 0$.
This proves Claim \ref{sNeg}.

\par \vskip 1pc
We continue the proof of Lemma \ref{HR}.
For $\R$-Cartier divisors $L_{1}$, $L_{2}$,
let $I(L_{1}, L_{2})$ be the intersection number $(H_{1} \cdots H_{n - 2} \cdot L_{1} \cdot L_{2})$.
Then, $I( , )$ defines a bilinear form on the real vector space
$\NS_{\R}(X)$.

We claim that the symmetric matrix corresponding to
$I(, )$ has exactly one positive eigenvalue.
Indeed, it has a positive eigenvalue, since $I(P, P) > 0$ for an ample divisor $P$.
If $L$ is an $\R$-Carier divisor such that $I(P, L) = 0$.
Then, $I(L, L) \leq 0$ by Claim \ref{sNeg}.
Hence, the other eigenvalues are non-positive, and the claim is proved.

We now prove Property (HI) by induction on $n$.
We may assume that $n \ge 3$.
By the arguments so far, it remains to show the assertion that $I(, )$ is non-degenerate.
Assume that Property (HI) holds in dimension $n-1$; especially the assertion holds for $n - 1$.
Let $L$ be an $\R$-Cartier divisor on $X$ such that
$I(L, N) = 0$ for every $\R$-Cartier divisor $N$.
We need to prove that $L \equiv 0$ (numerical equivalence).

Let $P$ be a very ample prime divisor.
Then,
$${H_1}_{\, |P} \cdots {H_{n-2}}_{\, |P} \cdot L_{\, |P} = I(L, P) = 0 .$$
By induction, Property (HI) holds in dimension $n-1$. Hence, either

\begin{itemize}
\item[(a)]
$L_{\, | P} \equiv 0$, or
\item[(b)] 
${H_1}_{\, |P} \cdots {H_{n-3}}_{\, |P} \cdot (L_{\, |P})^{2} < 0$, i.e.,
$H_{1} \cdots H_{n-3} \cdot P \cdot L^{2} < 0$.
\end{itemize}

Suppose that (b) holds for every very ample prime divisor $P$.
Since $H_{n -2}$ is an ample divisor, we have numerical equivalence
$H_{n - 2} \equiv \sum p_{i} P_{i}$ for some $p_{i} > 0$ and very ample prime divisors $P_{i}$.
Since (b) holds for every $P_i$,
we have $H_{1} \cdots H_{n-2} \cdot L^{2} < 0$.
This contradicts $I(L, L) = 0$.
Thus some very ample prime divisor $P'$ does not satisfy (b).

Let $C$ be an arbitrary curve on $X$.
Then, we can find a very ample prime divisor $P''$ in $|mP'|$
for some $m > 0$ such that $P''$ contains $C$.
Since $P''$ ($\sim mP'$) does not satisfy (b), it satisfies (a),
i.e., $L_{\, |P''} \equiv 0$.
Thus, $L \cdot C = L_{\, | P''} \cdot C = 0$. So $L \equiv 0$.
Therefore, $I(,)$ is non-degenerate. We have proved the assertion.
Hence Property (HI) holds for any $n \ge 2$. This also proves Lemma \ref{HR}.
\end{proof}

Let $G \le \Aut(X)$. We define the subset of {\it null-entropy elements} of $G$ as
$$N(G) := \{g \in G \, | \, \text{\rm the entropy} \, h(g) = 0\} .$$
$N(G)$ may not be a subgroup of $G$.

We quote the following result of \cite{DS04}.

\begin{lemma}\label{DS} (\cite{DS04})
Let $X$ be a normal projective variety
and $G \le \Aut(X)$. Suppose that
$G^* := G_{\, | \NS_{\R}(X)} \cong \Z^{\oplus r}$
and every non-trivial element of $G^*$ is of positive entropy. Then
there are nef $\R$-Cartier divisors $L_i$ ($1 \le i \le r+1$)
such that each $L_i$ is a common eigenvector of $G$ with
$$g^*L_i = \chi_i(g) L_i \,\,\, (g \in G)$$
for some $\chi_i(g) \in \R_{> 0}$, that
$$L_1 \dots L_{r+1} \ne 0$$
as an element of $H^{r+1, r+1}(X)$,
and that the homomorphism
$$\begin{aligned}
\varphi: G \, &\to \, (\R^r, +) \\
g \, &\mapsto \, (\log \chi_1(g), \dots, \log \chi_r(g))
\end{aligned}
$$
has the kernel equal to $N(G)$ and hence it
induces an isomorphism from $G/N(G)$ ($\cong G^*$) onto a spanning lattice
of $(\R^r, +)$.
\end{lemma}

\begin{proof}
This is proved in \cite[Theorems 4.3 and 4.7]{DS04}
by considering the action of $G$ on $\NS_{\R}(X)$
instead of that on $H^{1,1}(X)$. See also \cite[Theorems 1.1 and 1.2]{Z-Tits}.

For the isomorphism $G^* \cong G/N(G)$, we just need to check the assertion that
the natural representation 
$$G \to \GL(\NS_{\R}(X)), \hskip 1pc g \mapsto {g^*}_{\, | \NS_{\R}(X)}$$
has kernel $K$ equal to $N(G)$
(and image equal to $G^*$ by definition). Indeed, the kernel
$K \le \Aut_{[H]}(X)$
for any ample divisor class $[H]$,
and $\Aut_{[H]}(X)$ is a finite extension of the identity connected component $\Aut_0(X)$
of $\Aut(X)$ by \cite[Proposition 2.2]{Li}. So $K$ is virtually contained in $\Aut_0(X)$
and the latter continuous group acts on the integral lattice $\NS(X)/(\torsion)$ as identity.
Thus $K \le N(G)$. Conversely, since every $n \in N(G)$ or its action $n^*$ on $\NS_{\R}(X)$
is of null entropy, we have $n^* = \id$ on $\NS_{\R}(X)$ by the assumption on $G^*$. Hence $n \in K$.
So $N(G) \le K$.
We have proved the required assertion $K = N(G)$.
\end{proof}

\begin{setup}\label{bigA}
Applying Lemma \ref{DS} to our $G^* = G_{\, | \NS_{\R}(X)} \cong \Z^{\oplus n-1}$ in Hyp(A),
with $n = \dim X$,
we get nef $\R$-Cartier divisors $L_i$ ($1 \le i \le n$), which are common eigenvectors
of $G$.
Set $$A := L_1 + \cdots + L_n.$$
Then $(L_1 + \dots + L_n)^n \ge L_1 \cdots L_n$ and the latter is nonzero
in $H^{n, n}(X, \R) = \R$ by Lemma \ref{DS} and indeed positive, $L_i$ being nef.
Thus $A$ is a nef and big divisor.

Since $L_1 \cdots L_n$ is a scalar, for every $g \in G$, we have
$$L_1 \cdots L_n = g^*(L_1 \cdots L_n) = (\chi_1(g) \cdots \chi_n(g))(L_1 \cdots L_n) .$$
Hence
$$\chi_1 \cdots \chi_n = 1.$$
\end{setup}

\begin{lemma}\label{nefbig}
Suppose that $X$ and $G$ satisfy Hyp(A).
Then we have:
\begin{itemize}
\item[(1)]
The $A = L_1 + \cdots + L_n$ in \ref{bigA} is a nef and big $\R$-Cartier divisor.
\item[(2)]
For integer $k >> 1$, we have $A = A_k + E/k$
such that $A_k$ is an ample $\Q$-Cartier divisor and $E$ is a fixed effective $\R$-Cartier divisor.
\item[(3)]
Suppose further that $(X, D)$ is a klt pair with $D$ an effective $\R$-Cartier divisor.
Then, replacing $A_k$ by some $A_k'$ with $A_k' \sim_{\Q} A_k$ and $A$ by $A' := A_k' + E/k$
with $A' \equiv A$ (numerical equivalence),
we may assume that $A$ is an effective divisor and the pair $(X, D + A)$ has at worst klt singularities.
\end{itemize}
\end{lemma}

\begin{proof}
(1) has been proved in \ref{bigA}.
(2) is a consequence of (1)
and proved in \cite[Proposition 2.61]{KM} or \cite[II, Theorem 3.18]{ZDA} (for $\R$-divisors).
(3) is true, because $(X, D)$ is klt
and klt is an open condition. Indeed, we just take $k >>1$ and replace
$A_k$ by $A_k' := \frac{1}{m} B$ with $B$ a general member of $|mA_k|$ for some $m >> 1$.
See \cite[Corollary 2.35 (2)]{KM}.
\end{proof}

\begin{lemma}\label{per}
Assume that $X$ and $G$ satisfy Hyp(A). For the $A$ in \ref{bigA}, we have:
\begin{itemize}
\item[(1)]
Suppose that $Z \in H^{n-k, n-k}(X)$ is a $G$-periodic
class for some $n-k$ in $\{1, \dots, n-1\}$.
Then $A^k \cdot Z = 0$.
\item[(2)]
We have $A^{n-1} \cdot c_1(X) = 0$; and $A^{n-2} \cdot c_1(X)^2 = A^{n-2} \cdot c_2(X) = 0$ when $n \ge 3$.
Here
$c_i(X)$ denotes $i$-th Chern class; so $c_1(X) = [-K_X]$,
and $c_2(X)$ is regarded as a linear form on $n-2$ copies of $\NS_{\R}(X)$
as defined in \cite[p. 265]{SW}.
\item[(3)]
Suppose that $Z \subset X$ is a $G$-periodic positive-dimensional proper subvariety of $X$.
Then $A^{\dim Z} \cdot Z = 0$.
\end{itemize}
\end{lemma}

\begin{proof}
(2) and (3) are consequences of (1), noting that $g^*c_i(X) = c_i(X)$.

(1) was proved in \cite[Lemma 2.6]{max}.
To be precise, replacing $G$ by a finite-index subgroup, we may assume that $g^*Z = Z$ ($g \in G$).
Note that for all $i_j$, we have
\begin{equation}\label{(*)}
L_{i_1} \cdots L_{i_k} \cdot Z = 0 .
\end{equation}
Indeed,
since $\varphi(G) \subset \R^{n-1}$ is a spanning lattice, $k \le n-1$, and $\chi_1 \cdots \chi_n
= 1$, we can choose $g \in G$ such that
$\chi_{i_j}(g) > 1$ for all $i_j$.
Acting on the left hand side of the equality (\ref{(*)}) (a scalar) with $g^*$ and noting that $g^*Z = Z$,
we conclude the equality (\ref{(*)}). This, in turn, implies that $A^k \cdot Z = 0$,
since $A = \sum L_i$ and hence $A^k$ is the sum of $L_{i_1} \cdots L_{i_k}$.
\end{proof}

The following result was first proved by
Nakamaye for $\Q$-divisors, and generalized to the following for $\R$-divisors
by Ein - Lazarsfeld - Mustata - Nakamaye - Popa (see e.g. \cite[Example 1.11]{ELMNP}).
For a nef $\R$-Cartier divisor $L$ on a projective variety $X$, let
$$\Null(L) = \bigcup_{L_{\, |Z} \, \text{\rm not big}} \, Z$$
where $Z$ runs over all positive-dimensional subvarieties of $X$.
Since $L_{\, |Z}$ is nef, it is not big if and only if $L^{\dim Z} \cdot Z = 0$.

\begin{lemma}\label{ELMNP} (cf.~\cite{ELMNP})
Let $X$ be a normal projective variety and $L$ a nonzero nef $\R$-Cartier divisor.
Then $\Null(L)$ is a Zariski-closed proper subset of $X$.
(Indeed, $\Null(L)$ is equal to the augmented base locus $B_{+}(L)$
which we will not use in the sequel).
\end{lemma}

\begin{lemma}\label{null}
Suppose that $X$ and $G$ satisfy Hyp(A). For the $A$ in \ref{bigA},
we have
$$\Null(A) = \bigcup_{\text{Y is G-periodic}} \, Y$$
where $Y$ runs over all positive-dimensional $G$-periodic proper subvarieties of $X$.
In particular, $A$ is ample if and only if every $G$-periodic proper
subvariety of $X$ is a point.
\end{lemma}

\begin{proof}
The second assertion follows from the first, Lemma \ref{per}
and the Campana-Peternell generalization of Nakai-Moishezon ampleness criterion to $\R$-divisors.

For the inclusion ``$\supseteq$", we just copy the argument in the proof of Lemma \ref{per} (1).

The inclusion ``$\subseteq$" is as in \cite[Lemma 2.6]{max}. Indeed,
assume that $Y \subseteq \Null(A)$ with $0 < s := \dim Y$, so that
$A_{|Y}$ is not big, i.e., $A^s \cdot Y = 0$.
Write $A = A_k + E/k$ as in Lemma \ref{nefbig}.
Since $A_{|Y}$ is not big, $Y \subseteq \Supp E$.
Since $A = \sum L_i$ with $L_i$ nef and $A^s$ is the sum of $L_{i_1} \cdots L_{i_s}$,
the condition $A^s \cdot Y = 0$ means $L_{i_1} \cdots L_{i_s} \cdot Y = 0$ for all
$i_j$. Since $L_{i_j}$ are
all $g^*$-eigenvectors, reversing the process,
we get $A^s \cdot g(Y) = 0$ and hence $g(Y) \subset \Supp E$
by the above reasoning.
Now $Y$ is contained in the Zariski-closure $\overline{\cup_{g \in G} \, g(Y)}$.
This closure is $G$-stabilized and contained in $\Supp E$, and every irreducible
component of it is a positive-dimensional $G$-periodic proper subvariety of $X$.
Hence $Y$ is contained in the right hand side of the equality in Lemma \ref{null}.
This proves the inclusion ``$\subseteq$".
\end{proof}

\begin{lemma}\label{ample}
Theorem \ref{ThB} holds under Condition (iii), which is equivalent to the condition that
the $A$ in \ref{bigA} is an ample divisor (cf.~Lemma \ref{null}).
Further, we can take $Y = X$.
\end{lemma}

\begin{proof}
By Lemma \ref{per}, $K_X^i \cdot A^{n-i} = 0$ ($i = 1, 2$),
so $K_X \equiv 0$ by Lemma \ref{HR} and the ampleness of $A$.
The known abundance theorem in the case of zero Kodaira dimension implies
that $K_X \sim_{\Q} 0$. See [16, V. Corollary 4.9].

Let $m' > 0$ be the smallest integer such that
$m' K_X \sim 0$. Let
$$\pi : \hat{X} = Spec \, \oplus_{i=0}^{m'-1} \, \OO(-iK_X) \, \to \, X$$
be the Galois $\Z/(m')$-cover
which is called the {\it global index-1 cover} and is \'etale outside $\Sing X$
such that $K_{\hat X} \sim 0$ and $\hat{X}$ has at worst canonical singularities.
Since each $-iK_X$ is stabilized by $G$,
the action of $G$ on $X$ lifts to a faithful action of $G$ on $\hat{X}$.
By Lemma \ref{birAct}, $\hat{X}$ and $G$ satisfy Hyp(A), after replacing
$G$ by a finite-index subgroup.

Let ${\hat A}$ be the $\pi$-pullback of $A$. It is also a sum of common nef eigenvectors of $G$.
The ampleness of $A$
implies that of ${\hat A}$, since $\pi$ is a finite morphism.
By Lemma \ref{per} and since $n \ge 3$, we have ${\hat A}^{n-2} \cdot c_2(\hat{X}) = 0$.
This and the ampleness of ${\hat A}$
(and Miyaoka's pseudo-efffectivity of $c_2$)
imply that $c_2(\hat{X})$ is zero, as a linear form on the product of $(n-2)$-copies of $\NS_{\R}(\hat{X})$;
see \cite[pp.~265-267, Proposition 1.1]{SW}.
Since $\pi$ is $G$-equivariant,
every $G$-periodic proper subvariety of $\hat{X}$ is a point,
because the same holds on $X$ by assumption.
In particular, the singular locus $\Sing \hat{X}$ of $\hat{X}$ is isolated.

We now apply \cite[Theorem 1.16]{GKP} and deduce that
$\hat{X} = T'/F'$ for some
abelian variety $T'$ where the finite group $F'$ acts on $T'$ freely outside a codimension-2 subset.
Again $n \ge 3$ is used so that the condition ($0 =$) $\dim \Sing \hat{X} \le n-3$ in \cite{GKP} is satisfied.
Now $X$ is covered by the complex torus $T'$ via the composition 
$$T' \to T'/F = \hat{X} \to X$$ 
which is \'etale in codimension $1$.
By \cite[Sect. 3, especially Proposition 3]{Be} or \cite[\S 2.15]{max},
the assertions (1) and (3) in Theorem \ref{ThB} hold,
with $Y = X$. Indeed, since $\hat{X}$ has no positive-dimensional proper subvariety
which is periodic under the action of $G$,
so do $X$ and $T$ under the action of $G$ and $\widetilde{G}$.
Since $\widetilde{G}$ normalizes $F := \Gal(T/X)$, it stabilizes the subset of $T$ where $F$
does not act freely, so the latter subset is finite or empty.

The assertion (2) is true because $K_T \sim 0$. The others are void since $Y = X$.
\end{proof}

The following is the key step towards the proof of Theorem \ref{ThB}.
Hyp(wB) is a weaker form of Hyp(B), i.e., without the pseudo-effectivity of $K_X + D$.

\begin{proposition}\label{PropA}
Assume that $X$ and $G$ satisfy Hyp(A), and
Hyp(wB): for some effective $\R$-divisor $D$ whose irreducible components are $G$-periodic,
the pair $(X, D)$ has at worst $\Q$-factorial klt singularities.
Let $A = \sum L_i$ be the sum in \ref{bigA} of nef $\R$-divisors $L_i$
which are also $G$-eigenvectors such that $A$ is a nef and big divisor.
Replacing $G$ by a finite-index subgroup and $A$ by a large multiple,
the following are true.
\begin{itemize}
\item[(1)]
There is a sequence $\tau_s \circ \cdots \circ \tau_0$ of birational maps:
$$X = X(0) \overset{\tau_0}\dashrightarrow X(1) \overset{\tau_1}\dashrightarrow \cdots
\dashrightarrow X(s) \overset{\tau_s}\to X(s+1) = Y$$
such that
each $X(j) \dasharrow X(j+1)$
($0 \le j < s$) is either a divisorial contraction (and hence a morphism)
of a $(K_{X(j)} + D(j))$-negative extremal ray
or a $(K_{X(j)} + D(j))$-flip (and hence an isomorphism in codimension-$1$);
here $D(i) \subset X(i)$ ($0 \le i \le s+1$) is the direct image of $D$.
The $\tau_s$ is a birational morphism.
So the composition
$X \dashrightarrow Y$ is surjective in codimension-$1$.
\item[(2)]
The pair $(X(i), D(i) + A(i))$ ($0 \le i \le s+1$) and hence the pair $(X(i), D(i))$
have at worst klt singularities (cf.~\cite[Proposition 2.41, Corollary 2.39]{KM}).
$X(j)$ ($0 \le j \le s$) is $\Q$-factorial.
\item[(3)]
The induced action of $G$ on each $X(i)$ ($0 \le i \le s+1$) is biregular.
$X(i)$ and $G$ satisfy Hyp(A).
\item[(4)]
The direct image $L(i)_j$ on $X(i)$ of $L_j$ is a nef $\R$-Cartier $G$-eigenvector.
Hence the direct image $A(i)$ on $X(i)$ of $A$ is a nef and big $\R$-Cartier divisor.
Further, $L(i)_j = \tau_i^*L(i+1)_j $, so $A(i) = \tau_i^*A(i+1)$.
\item[(5)]
$K_Y + D_Y$ is an $\R$-Cartier divisor, and
$K_{X(s)} + D(s) = \tau_s^*(K_Y + D_Y)$, where $D_Y := D(s+1) = {\tau_s}_*D(s)$.
\item[(6)]
$K_Y + D_Y + A_Y$ is an ample divisor, where $A_Y := A(s+1) = \tau_*A(s)$.
Hence $K_{X(s)} + D(s) + A(s) = \tau_s^*(K_Y + D_Y + A_Y)$ is a nef and big divisor.
\item[(7)]
The union of all positive-dimensional $G$-periodic proper subvarieties of
$X(i)$ ($0 \le i \le s+1$) coincides with $\Null(A(i))$
and hence a Zariski-closed proper subset of $X(i)$. Further,
$A(i)_{\, |Z} \equiv 0$ (numerical equivalence) for every positive-dimensional
subvariety $Z \subseteq \Null(A(i))$; this is especially true when $Z$ is a component of $D(i)$.
\end{itemize}
\end{proposition}

The rest of the section is devoted to the proof of Proposition \ref{PropA}.

\begin{setup}\label{MMP}
{\rm
For the convenience of the readers, we recall the traditional LMMP as in \cite[\S 3.7]{KM},
before running some {\it directed} LMMP for the klt pair $(X, D + A)$ chosen in Lemma \ref{nefbig}.
If $K_X + D + A$ is already nef, then $(X, D + A)$ is the end product and we stop the LMMP.
Suppose that $K_X + D + A$ is not nef. By the cone theorem (\cite[Theorem 3.7]{KM}),
the following closed cone of effective $1$-cylces
$$\NE(X)$$
contains a $(K_X + D + A)$-negative extremal ray
$R = \R_{> 0}[\ell]$ generated by an (extremal) rational curve $\ell$.
There is a corresponding extremal contraction
$$f := \Contr_R : X \to Y$$
onto a normal projective
variety $Y$ such that fibres of $f$ are connected, and a curve $C \subset X$
is contracted by $f$ to a point on $Y$ if and only if the class $[C] \in R$.
Further, if $B$ is a Cartier divisor on $X$ such that the intersection $B \cdot \ell = \deg(B_{| \ell}) = 0$
then $B$ equals the total transform (or pullback) $f^*B_Y$ for some Cartier divisor $B_Y$ on $Y$.

There are $4$ possible cases.

\par \vskip 1pc
Case (I) (Fano type) $\dim Y = 0$, i.e., $Y$ is a point. Then the Picard number $\rho(X) = 1$
and $-(K_X + D + A)$ is an ample divisor on $X$. Hence $X$ is of Fano type.

Case (II) (Fano fibration) $0 < \dim Y < \dim X$. Let $F$ be a general fibre of $f$.
Then $-(K_X + D + A)$
is positive on every curve $C \subset F$ and in fact, $-(K_X + D + A)_{| F} = -(K_F + (D + A)_{|F})$
is an ample divisor on $F$. So $F$ is of Fano type.
This $f$ is called a Fano-fibration.

Case (III) (divisorial) $\dim Y = \dim X$ and the {\it exceptional locus}
$\Exc(f)$ of $f$
(the subset of $X$ of point at which $f$ is not an isomorphism)
is a (necessarily) irreducible divisor. In this case
$f$ is a birational morphism. The Picard numbers satisfy $\rho(Y) = \rho(X) - 1$.
Set $X(1) := Y$ and let 
$$D(1) := f_*D, \hskip 1pc A(1) := f_*A$$ 
be the direct images of $A, D$, respectively.

Case (IV) (flip) $\dim Y = \dim X$ and the exceptional locus $\Exc(f)$ of $f$ is a Zariski-closed subset of $X$
of codimension $\ge 2$ in $X$. This $f : X \to Y$ is called a flipping contraction.
In this case, the natural map
$$X^+ := \Proj \, \oplus_{m \ge 0} \, \OO_Y({f}_*\lfloor m(K_X + D + A) \rfloor) \, \to \, Y$$
is a birational morphism such that $K_{X^+} + D^+ + A^+$ is relatively ample over $Y$.
Here $\lfloor m(K_X + D + A) \rfloor$ is the integral part (or the round down)
of the $\R$-divisor $m(K_X + D + A)$;
$D^+ \subset X^+$ and $A^+ \subset X^+$ are the proper transforms of $D$ and $A$.
Both birational maps $X \to Y$ and $X^+ \to Y$ are isomorphisms in codimension-$1$.
In particular, we have the identification of $\NS_{\R}(X) = \NS_{\R}(X^+)$.
The map $X \dasharrow X^+$ is called a $(K_X + D + A)$-flip.
Set 
$$X(1) := X^+, \,\, D(1) := D^+, \,\, A(1) := A^+ .$$

\par \vskip 1pc
Set 
$$X(0) := X, \,\, D(0) := D, \,\, A(0) := A .$$
Since $(X, D + A)$ has only
$\Q$-factorial klt singularities so is the new pair $(X(1), D(1) + A(1))$
in Case (III) or (IV) by the LMMP.

If Case (I) or (II) occurs, we got the end product and stop the LMMP.
If Case (III) or (IV) occurs, we apply the LMMP to $(X(1), D(1) + A(1))$
and get divisorial contraction
$$\tau_1 : (X(1), D(1) + A(1)) \to (X(2), D(2) + A(2))$$ 
with $\rho(X(2)) = \rho(X(1)) - 1$,
or flip 
$$\tau_1 : (X(1), D(1) + A(1)) \dashrightarrow (X(2), D(2) + A(2)) = (X(1)^+, D(1)^+ + A(1)^+)$$ 
with $\rho(X(2)) = \rho(X(1))$,
or Fano type, or Fano fibration.
Since the Picard number $\rho(X(0))$ is finite, Case (III) can only occur finitely many times.
Thus there exist a sequence of birational maps
$$X = X(0) \overset{\tau_0}\dashrightarrow X(1) \overset{\tau_1}\dashrightarrow \cdots
\dashrightarrow X(s) \overset{\tau_s}\dashrightarrow \cdots $$
and an integer $s \ge 0$ such that every
$$\tau_i : (X(i), D(i) + A(i)) \dashrightarrow (X(i+1), D(i+1) + A(i+1))$$
($i < s$) is divisorial or flip, and either there is a contraction of Fano type or Fano fibration
on $(X(s), D(s) + A(s))$, or every contraction $\tau_j$ ($j \ge s$) is a flip. Here
$D(i) \subset X(i)$ and $A(i) \subset X(i)$
are the direct images of $D = D(0)$ and $A = A(0)$.

In dimension $\le 3$, the LMMP terminates. In higher dimensions, the termination is not known, i.e.,
the termination conjecture below is not proven yet: there is no infinite sequence of flips.
However, the so called directed-flip, or $(K_X + D + A)$-MMP with scaling
terminates \cite[Corollary 1.4.2]{BCHM}.
This is one of the key ingredients of our proof.
}
\end{setup}

We continue the proof of Proposition \ref{PropA}.
We follow the procedures in \ref{MMP} and some steps of \cite{max},
but we need to take care of the $G$-equivariance of the LMMP.

{\it Replacing $G$ by a finite-index subgroup, we may and will assume that every
irreducible component of $D$ is stabilized by $G$.}

\begin{lemma}\label{nef}
In \ref{MMP}, we can choose the LMMP sequence of divisorial contractions or flips
$\tau_j : X(j) \dashrightarrow X(j+1)$
($j < s$) such that the following are true for all $i \in \{0, 1, \dots, s\}$.
\begin{itemize}
\item[(1)]
The induced action of $G$ on $X(i)$ is biregular.
The action of $G$ on $X(i)$ satisfies Hyp(A).
\item[(2)]
The direct image $L(i)_j \subset X(i)$ of each $L_j \subset X$ is nef and a common eigenvector of $G$,
so the direct image $A(i) = \sum_{j=1}^n L(i)_j$ on $X(i)$ of $A$ is nef and big.
\item[(3)]
$K_{X(s)} + D(s) + bA(s)$ is a nef and big divisor for some (and hence all) $b >> 1$.
\end{itemize}
\end{lemma}

\begin{proof}
Since $A$ is big, the bigness of the direct image $A(i)$ of $A$ is clear.
Also the nefness of $A(i)$ would follow from that of $L(i)_j$ for all $j$.
Write $A = A_k + E/k$ as in Lemma \ref{nefbig} such that $(X, D + A)$ is klt.
Fix an ample Cartier divisor $M$ such that the pair
$(X, D + A + M)$ is nef and klt. This is doable because klt is an open condition
and we can replace $M$ by $\frac{1}{c}M'$ with $M'$ a general member of $|cM|$ for some large integer $c$.

The bigness in (3) is clear for $b >> 1$.
So assume the contrary that $K_X + D + bA$ is not nef for any $b > 1$.
We now consider $K_X + D + A$, but $A$ may be replaced by $bA$ for some $b >> 1$.
Since our boundary divisor $A$ is larger than some ample divisor $A_k$,
there are only finitely many $(K_X + D + A)$-negative extremal rays $\R_{> 0} [\ell]$ in $\NE(X)$.
See the cone theorem \cite[Theorem 3.7]{KM} or \cite[Corollary 3.8.2]{BCHM}.
We may assume that all these $\ell$ satisfy $A \cdot \ell = 0$
and $(K_X + D) \cdot \ell < 0$,
otherwise, we would have $(K_X + D + A) \cdot \ell > 0$ for all these finitely many $\R_{> 0}[\ell]$
and hence $K_X + D + A$ is nef, after replacing $A$ by a large multiple.
Since $A = \sum_{i=1}^n L_i$ and each $L_i$ is nef,
$A \cdot \ell = 0$ means $L_i \cdot \ell = 0$ for all $i$.
Since $L_i \cdot g_*\ell = \chi_i(g) L_i \cdot \ell = 0$ and hence $A \cdot g_*\ell = 0$, and
$$(K_X + D) \cdot g_*\ell = g^*(K_X + D) \cdot \ell = (K_X + D) \cdot \ell < 0$$
this $g_*\ell$ ($= g(\ell)$ as a set, also an extremal curve) satisfies the same conditions as $\ell$. So
these finitely many extremal rays $\R_{> 0} [\ell]$ are permuted by $G$
and may assumed to be fixed by $G$, after replacing $G$ by a finite-index subgroup.

Now we run the $M$-directed LMMP for the pair $(X, D + A)$ as in \cite[Definition 2.4]{Bi}.
Remember that we may assume that $K_X + D + A$ is not nef while $K_X + D + A + M$ is nef
by the choice of $M$.
There is an extremal ray $R = \R_{> 0} [\ell]$ such that
for
$$\lambda_0 := \inf \{\alpha \ge 0 \, | \, K_X + D + A + \alpha M \,\, \text{\rm is nef} \}$$
we have that
$K_X + D + A + \lambda_0 M$ is nef, $(K_X + D + A) \cdot \ell < 0$
and $(K_X + D + A + \lambda_0 M) \cdot \ell = 0$.
As mentioned above, such $\ell$ satisfies $A \cdot \ell = 0$ and hence $L_i \cdot \ell = 0$,
$(K_X + D) \cdot \ell < 0$,
and we may assume that the extremal ray $R$ is $G$-stable.
Let 
$$f = \Contr_R: X \to Y$$ 
be the extremal contraction
as in \ref{MMP}. Since $R$ is $G$-stable and $f$ is completely determined by $R$,
the action of $G$ on $X$ descends to a biregular action of $G$ on $Y$ such that the morphism
$f : X \to Y$ is $G$-equivariant. We consider the $4$ cases in \ref{MMP} separately.

\par \vskip 1pc
Case (I) (Fano type). Since the Picard number $\rho(X) = 1$,
we take $H$ to be an ample generator of $\NS(X)/(\torsion) \cong \Z$.
Then we have $\Aut(X) \le \Aut_{[H]}(X)$.
Hence, using \cite{Li} as in the proof of Lemma \ref{birAct},
$\Aut(X)$ is a finite extension of the identity component $\Aut_0(X)$,
so $\Aut(X)$ and hence $G$ are of null entropy. This contradicts the assumption that
every element of $G \setminus \{\id\} = \Z^{\oplus n-1} \setminus \{\id\}$
is of positive entropy and that $n - 1 \ge 2$.

Case (II) (Fano fibration). Since $f : X \to Y$ is a non-trivial $G$-equivariant
fibration, $\rank(G) \le \dim X - 2 = n-2$, by \cite[Lemma 2.10]{Z-Tits}.
This contradicts the fact $\rank(G) = n-1$.

Case (III) (divisorial).
By Lemma \ref{birAct}, $Y$ and $G$ satisfy Hyp(A).
Since $L_j \cdot \ell = 0$, a property of the contraction $f$
implies that $L_j = f^*L(1)_j$ for some $\R$-Cartier divisor $L(1)_j$ on $Y$.
See \cite[Lemma 3-2-5]{KMM}.
We have $L(1)_j = f_*L_j$ by the projection formula.
Since $L_j$ is nef, so is $L(1)_j$. Clearly, if $g^*L_j = \chi(g)L_j$
then $g^*L(1)_j = \chi(g) L(1)_j$.
Set 
$$(X(1), D(1) + A(1)) := (Y, f_*(D + A)) .$$
We can continue the $G$-equivariant LMMP.

Case (IV) (flip). Since the flip $X^+$ is uniquely determined by the extremal ray $R$
(which is stabilized by $G$),
there is a biregular action of $G$ on $X^+$.
By Lemma \ref{birAct}, $X^+$ and $G$ satisfy Hyp(A).
As in Case (III), we have $L_j = f^*f_*L_j$
and $f_*L_j$ is an $\R$-Cartier nef divisor on $Y$.
Let $L(1)_j$ be the pullback by the birational morphism $X^+ \to Y$.
Then it is a $\R$-Cartier nef divisor.
Clearly, if $g^*L_j = \chi(g)L_j$
then $g^*L(1)_j = \chi(g) L(1)_j$.
Set 
$$(X(1), D(1) + A(1)) := (X^+, D^+ + A^+) .$$
We can continue the $G$-equivariant LMMP.

\par \vskip 1pc
Since we have settled the cases of Fano type and Fano fibration and since a divisorial contraction
decreases the Picard number, setting 
$$(X, D + A) = (X(0), D(0) + A(0))$$
we may assume that we have a sequence
$$(X(0), D(0) + A(0)) \overset{\tau_0}\dashrightarrow (X(1), D(1) + A(1))
\overset{\tau_1}\dashrightarrow \cdots
\dashrightarrow (X(s), D(s) + A(s)) \overset{\tau_s}\dashrightarrow \cdots $$
of flips $\tau_i$ corresponding to the extremal rays
$\R_{> 0}[\ell_i]$ with $(K_{X(i)} + D(i) + A(i)) \cdot \ell_i < 0$
and $(K_{X(i)} + D(i) + A(i) + \lambda_i M(i)) \cdot \ell_i = 0$.
Here $D(i)$, $A(i)$ and $M(i)$ on $X(i)$
are the direct images of $D$, $A$ and $M$ (which is the fixed ample divisor on $X$).
The real number $\lambda_i$ satisfies:
$$\lambda_i = \inf \{\alpha \ge 0 \, | \, K_{X(i)} + D(i) + A(i) + \alpha M(i) \,\, \text{\rm is nef} \}.$$

If $K_{X(s)} + D(s) + A(s)$ is nef for some $s$, then we are done.
Suppose this is not the case for any $s \ge 0$. Then we have an infinite sequence of flips
$$\tau_i : (X(i), D(i) + A(i)) \dashrightarrow (X(i+1), D(i+1) + A(i+1))$$ 
as above.

Since our pair $(X, D + A)$ is $\Q$-factorial klt and
the boundary divisor $D + A$ is a big divisor and larger than an ample divisor $A_k$
as in Lemma \ref{nefbig}, we can apply \cite[Corollary 1.4.2]{BCHM} or \cite[Theorem 1.9 (i)]{Bi}.
Thus the above LMMP must terminate. Hence $K_{X(s)} + D(s) + A(s)$ must be nef for some $t \ge 0$.
This proves Lemma \ref{nef}.
\end{proof}

We continue the proof of Proposition \ref{PropA}. Replacing $X$ by the $X(s)$ in Lemma \ref{nef},
we may assume that we already have $K_X + D + A$ nef and big for the $A$ in Lemma \ref{nefbig}.
By the base point free theorem (\cite[Theorem 3.3]{KM}, or \cite[Theorem 3.9.1]{BCHM}),
there exist a birational morphism $\gamma: X \to Y$ onto a normal projective variety $Y$
and an ample $\R$-Cartier divisor $B$ on $Y$ such that
$K_X + D + A = \gamma^*B$. By the projection formula, $B = K_Y + D_Y + A_Y$
where $D_Y = \gamma_*D$ and $A_Y = \gamma_*A$
are the direct images. So
$$K_X + D + A = \gamma^*(K_Y + D_Y + A_Y) .$$

\begin{lemma}\label{crep}
Replacing $G$ by a finite-index subgroup and $A$ by a large multiple, we have:
\begin{itemize}
\item[(1)]
The action $G$ on $X$ descends to a biregular action on $Y$ so that $\gamma : X \to Y$
is $G$-equivariant.
\item[(2)]
$K_Y + D_Y$ is $\R$-Cartier and $K_X + D = \gamma^*(K_Y + D_Y)$.
The pair $(Y, D_Y + A_Y)$ and hence $(Y, D_Y)$ have only klt singularities
(cf.~\cite[Proposition 2.41, Corollary 2.39]{KM}).
\item[(3)]
$L_j = \gamma^*L(Y)_j$ ($1 \le j \le n$) for some $\R$-Cartier nef divisor $L(Y)_j$ on $Y$
which is also a common eigenvector of $G$.
Hence $A_Y = \gamma_*A = \sum_{j=1}^n L(Y)_j$ is a nef and big $\R$-Cartier divisor, and $A = \gamma^*A_Y$.
\item[(4)]
$Y$ and $G$ satisfy Hyp(A). Our
$Y$, $G$ and $A_Y = \sum L(Y)_j$ satisfy all the properties in
Lemmas \ref{DS}, \ref{nefbig} and \ref{per} for $X$, $G$ and $A = \sum L_j$.
\end{itemize}
\end{lemma}

\begin{proof}
Since $A = A_k + E/k$ with $A_k$ ample, the extremal rays $\R_{> 0} [\ell]$ contracted by $\gamma$
(i.e., perpendicular to the nef and big divisor $K_X + D + A = \gamma^*(K_Y + D_Y + A_Y)$)
satisfy
$$0 = \gamma^*(K_Y + D_Y + A_Y) \cdot \ell =
(K_X + D_Y + E/k + \varepsilon A_k) \cdot \ell + (1-\varepsilon) A_k \cdot \ell
> (K_X + D_Y + E/k + \varepsilon A_k) \cdot \ell $$
and hence are all $(K_X + D_Y + E/k + \varepsilon A_k)$-negative.
Thus there are only finitely many of such $\R_{> 0} [\ell]$ by the bigness of the (new)
boundary divisor $D_Y + E/k + \varepsilon A_k$.
See the cone theorem \cite[Theorem 3.7]{KM}.
Note that $\gamma : X \to Y$ is also a birational contraction of a
$(K_X + D_Y + E/k + \varepsilon A_k)$-negative
extremal face; see the contraction theorem \cite[Theorem 3-2-1]{KMM}.
If $\gamma : X \to Y$ is an isomorphism, then the lemma is clear.
Otherwise, there are such $\ell$;
replacing $A$ by a large multiple, we may assume that the above finitely many
extremal rays $\R_{> 0}[\ell]$ (contracted by $\gamma$)
satisfy
$$0 = (K_X + D + A) \cdot \ell = (K_X + D) \cdot \ell = A \cdot \ell .$$
Since $A = \sum_{j=1}^n L_j$ is a sum of nef divisors, we have $L_j \cdot \ell = 0$.

Note that if $\ell$ satisfies the condition
$(K_X + D) \cdot \ell = L_j \cdot \ell = 0$, then so are the $G$-images of $\ell$.
Since there are only finitely many of such extremal rays $\R_{> 0}[\ell]$,
we may assume that all of them are $G$-stable, after replacing $G$ by a finite-index subgroup.
This and the fact that the map $\gamma$ is the contraction of the extremal face generated by these extremal rays,
imply that $G$ descends to an action on $Y$ such that $\gamma$ is $G$-equivariant.
By Lemma \ref{birAct}, $Y$ and $G$ satisfy Hyp(A).

Now the cone theorem (\cite[Theorem 3.7]{KM} or \cite[Theorem 3-2-1, or Lemma 3-2-5]{KMM})
and $L_j \cdot \ell = 0$ for all the above extremal $\ell$
imply that $L_j = \gamma^*L(Y)_j$ for some $\R$-Cartier divisor $L(Y)_j$ on $Y$
which is a nef common eigenvector of $G$, since so is $L_j$.
Further, $L(Y)_j = \gamma_*L_j$ by the projection formula.
By the same cone theorem, $(K_X + D) \cdot \ell = 0$ implies that
$K_Y  + D_Y = \gamma_*(K_X + D)$ is a $\R$-Cartier divisor and $K_X + D = \gamma^*(K_Y + D_Y)$.

Since $(X, D + A)$ is klt, so is $(Y, D_Y + A_Y)$ (and hence $(Y, D_Y)$)
by the display preceding Lemma \ref{crep}.
This proves Lemma \ref{crep}.
\end{proof}

\par \vskip 1pc
We continue the proof of Proposition \ref{PropA}. Setting $\tau_s := \gamma$, by the arguments so far,
it remains to prove Proposition \ref{PropA} (7).
Its first part is a consequence of Lemmas \ref{ELMNP} and \ref{null}, thanks to (3).
For the second part of (7), we have only to prove it on $Y$,
since the sequence $X \dashrightarrow Y$ is $G$-equivalent.
Indeed, if $\tau_i : X(i) \to X(i+1)$ is divisorial, then $A(i) = \tau_i^*A(i+1)$, hence
$\Null(A(i))$ is just
the inverse of $\Null(A((i+1))$, using Lemma \ref{null}.
If $X(i) \to Y'$ and $X(i)^+ = X(i+1) \to Y'$ are the flipping contractions,
then $A(i)$ and $A(i+1)$ are the pullbacks of some nef and big divisor $A_{Y'}$ on $Y'$, so
$\Null(A(i))$ is
the inverse of $\Null(A_{Y'})$ by using Lemma \ref{null} and
noting that $Y'$ and $G$ also satisfy Hyp(A) by Lemma \ref{birAct}.

For the second part of (7), to prove the vanishing of ${A_Y}_{\, |Z}$ on $Y$,
by Lemma \ref{per}, we may assume that $k : = \dim Z \ge 2$.
As in \cite[Lemma 2.9]{max},
we prove first:
\begin{equation}\label{(*1)}
((K_Y + D_Y + A_Y)_{\, |Z})^{k-1} \cdot {A_Y}_{\, |Z} = (K_Y + D_Y + A_Y)^{k-1} \cdot A_Y \cdot Z = 0
\end{equation}
Indeed, since $A_Y = \sum L(Y)_i$, the above mid-term is
the summation of the following terms
$$
(K_Y + D_Y)^{k-1-t} \cdot L(Y)_{j_1}  \cdots L(Y)_{j_t} \cdot L(Y)_i \cdot Z
$$
where $0 \le t \le k-1 \le n-2$.
Now the vanishing of each term above can be verified as
in Lemma \ref{per}, since $g^*(K_Y + D_Y) = K_Y + D_Y$ for $g \in G$.
The equality (\ref{(*1)}) above is proved.

The equality (\ref{(*1)}) and ampleness of $K_Y + D_Y + A_Y$, restricted to
$Z$, imply that
$$((K_Y + D_Y + A_Y)_{|\, Z})^{k-2} \cdot ({A_Y}_{\, |Z})^2$$
is a non-positive scalar by Lemma \ref{HR},
and hence is zero since $K_Y + D_Y + A_Y$ and $A_Y$ are nef.
Thus ${A_Y}_{\, |Z} \equiv 0$ by Lemma \ref{HR}.
This completes the proof of Proposition \ref{PropA}.

\section{Proofs of Theorems \ref{ThB} and \ref{CorA} and Proposition \ref{PropE}}

{\it We first prove Theorem \ref{ThB}.}
Theorem \ref{ThB} is true under Condition (iii) by Lemma \ref{ample}.
Assume Condition (i). If $X$ is not uniruled, then (so is its resolution and hence)
$K_X$ is a pseudo-effective divisor by the uniruledness criterion due to Miyaoka-Mori and
Boucksom-Demailly-Paun-Peternell. Thus Hyp(B) holds by letting $D = 0$, so Condition (ii) holds.
If $X$ is uniruled,
the action of $G$ on $X$ descends to a biregular action on the base
of the {\it special} MRC (maximal rationally connected) fibration constructed in
\cite[Theorem 4.18]{IntS}, with general fibres rationally connected varieties.
The maximality of $\rank(G) = n-1$
and \cite[Lemma 2.10]{Z-Tits}
imply that this $G$-equivariant MRC must be trivial.
Since $X$ is uniruled and hence the general fibre is not a point, the triviality means
that the base is a point.
So $X$ is rationally connected, contradicting Condition (i).

\par \vskip 1pc
{\it From now on, we will prove Theorem \ref{ThB} under Condition (ii).}

Since $(X, D)$ is klt, if we set 
$D^{\varepsilon} := (1 + \varepsilon) D$
then $(X, D^{\varepsilon})$ is still klt for small $\varepsilon \in (0, 1)$;
see \cite[Corollary 2.35]{KM}.
Choose $A$ as in Lemma \ref{nefbig} such that $(X, D^{\varepsilon} + A)$ is klt.
Applying Proposition \ref{PropA} to the pair $(X, D^{\varepsilon} + A)$,
we get birational maps
$$\tau_i : X(i) \dashrightarrow X(i+1)$$ 
where
$\tau_j$ ($0 \le j < s$) is either a divisorial contraction or a flip, corresponding
to a $(K_{X(j)} + D^{\varepsilon}(j))$-negative extremal ray,
and 
$$\tau_s : X(s) \to X(s+1) = Y$$ 
is a birational morphism.
Here we let 
$$D(i), \,\, D^{\varepsilon}(i), \,\, A(i)$$ 
on $X(i)$ be the direct
images of $D$, $D^{\varepsilon}$, $A$, respectively;
note that $D^{\varepsilon}(i) = (1 + \varepsilon) D(i)$.
The assertions (4) - (6) follow from Proposition \ref{PropA}.

Next we show (7).
Since $K_X + D$ and hence $K_X + D^{\varepsilon}$ are pseudo-effective, so are their direct images
$K_{X(s)} + D(s)$ and $K_{X(s)} + D^{\varepsilon}(s)$.
To simplify the notation, we use 
$$\gamma : X \to Y$$ 
to denote the $\tau_s : X(s) \to X(s+1) = Y$
in Proposition \ref{PropA},
and let 
$$D = D(s), \,\, D^{\varepsilon} = D^{\varepsilon}(s) = (1 + \varepsilon)D, \,\, A = A(s) .$$
By the $\sigma$-{\it decomposition} for pseudo-effective divisors in \cite[III. \S 1.12]{ZDA},
$$K_{X} + D^{\varepsilon} = P + N$$
where $P$ is in the {\it closed movable cone} (generated by fixed-component free
Cartier divisors) and $N$ is an effective divisor.

{\it Replacing $G$ by a finite-index subgroup, we may and will assume that every $G$-periodic
divisor on $X$ is stabilized by $G$}. See Proposition \ref{PropA} (7).

\begin{claim}\label{N1}
We have $P \equiv 0$, so $K_{X} + D^{\varepsilon} \equiv N$.
Every irreducible component of $N$ is stabilized by $G$.
\end{claim}

\begin{proof}
We prove Claim \ref{N1}.
The uniqueness of the $\sigma$-decomposition implies
the assertion that $g^*P \equiv P$ and $g^*N = N$
for any $g \in G$
(so every component of $N$ is $G$-periodic and hence stabilized by $G$).
Indeed, 
$$K_{X} + D^{\varepsilon} = g^*(K_{X} + D^{\varepsilon}) = g^*P + g^*N$$
and $(K_{X} + D^{\varepsilon}) - g^*N = g^*P$ is movable, so
$g^*N \ge N$ by the minimality of the `negative part' $N$;
see \cite[III, Proposition 1.14]{ZDA}. Applying the above to $g^{-1}$
we get $N \ge g^*N$. The assertion follows.

Now
$$P \cdot M \cdot M_1 \cdots M_{n-2} = P_{\, | M} \cdot {M_1}_{\, | M} \cdots {M_{n-1}}_{\, |M}
\ge 0$$
for every irreducible divisor $M$
and nef $\R$-Cartier divisors $M_i$ because $P_{\, |M}$ is a pseudo-effective divisor on $M$;
also $P \cdot A^{n-1} = 0$ for the nef and big divisor $A$,
by Lemma \ref{per}. Thus we may apply \cite[Lemma 2.2]{nz2}
(by reducing to the hard Lefschetz theorem) to deduce that $P \equiv 0$.
This proves Claim \ref{N1}.
\end{proof}

\begin{claim}\label{Ns}
$N = 0$. Hence $K_{X} + D^{\varepsilon} \sim_{\Q} 0$.
\end{claim}

\begin{proof}
We prove Claim \ref{Ns}.
By Proposition \ref{PropA}, $K_{X} + D^{\varepsilon} + A$ is a nef and big divisor; further,
$$(K_{X}+ D^{\varepsilon} + A)_{\, | Z} = (K_{X} + D^{\varepsilon})_{\, |Z} \equiv N_{\, |Z}$$
is also nef for every $G$-periodic proper subvariety $Z$ of $X$,
especially for $Z = N_i$, a component of $N$.
Thus
$$N \cdot M \cdot M_1 \cdots M_{n-2} =  N_{\, | M} \cdot {M_1}_{\, | M} \cdots {M_{n-1}}_{\, |M}
\ge 0$$
for every irreducible divisor $M$
and nef $\R$-Cartier divisors $M_i$ because $N_{\, |M}$ is a pseudo-effective divisor on $M$
(even when $M = N_i$);
also $N \cdot A^{n-1} = 0$ for the nef and big divisor $A$,
by Lemma \ref{per}. As in Claim \ref{N1}, applying \cite[Lemma 2.2]{nz2}, we get $N \equiv 0$,
so the effective divisor $N = 0$.
Thus $K_{X} + D^{\varepsilon} \equiv N = 0$.
Since $(X, D^{\varepsilon})$ is klt by Proposition \ref{PropA},
the known abundance theorem in the case of zero log Kodaira dimension implies
that $K_{X} + D^{\varepsilon} \sim_{\Q} 0$. See \cite[V. Corollary 4.9]{ZDA}.
This proves Claim \ref{Ns}.
\end{proof}

\par \vskip 1pc
We return to the proof of Theorem \ref{ThB}.
Now $0 \equiv K_{X} + D^{\varepsilon} = (K_{X} + D) + \varepsilon D$.
This and the pseudo-effectivity of $K_{X} + D$ as assumed,
imply that $K_{X} + D \equiv 0$ and $\varepsilon D = 0$.
Hence the effective divisor $D = 0$, so $D^{\varepsilon} = (1+\varepsilon)D = 0$.
Thus $K_{X} \sim_{\Q} 0$ by Claim \ref{Ns}.

Note that in this proof of Theorem \ref{ThB}, we have been applying Proposition \ref{PropA}
to the pair $(X, D^{\varepsilon})$ and use $\gamma : X \to Y$ to denote
$\tau_s : X(s) \to X(s+1) = Y$ in Proposition \ref{PropA}.
The latter says that
$K_Y + \gamma_*D^{\varepsilon}$ is an $\R$-Cartier divisor (with
$\gamma_*D^{\varepsilon} = 0$ now) and it has
$K_X + D^{\varepsilon}$ ($ = K_{X}$ now) as its pullback by $\gamma$.
Thus $0 \sim_{\Q} K_{X} = \gamma^*K_Y$. Hence $K_Y \sim_{\Q} 0$.
This proves the assertion (7).

By Proposition \ref{PropA}, we have the ampleness of $K_Y + \gamma_*D^{\varepsilon} + A_Y$ ($\equiv A_Y$ now).
Thus we can apply Lemma \ref{ample} to $Y$ and $G$ (cf.~Proposition \ref{PropA} (2) and (3)).
In particular, assertions (1) - (3) are true.
This proves Theorem \ref{ThB} under Condition (ii).

We have completed the proof of Theorem \ref{ThB}.

\par \vskip 1pc
{\it Next we prove Main Theorem \ref{CorA}.}
Condition (1) implies Condition (2) by Theorem \ref{ThB} with $W := Y$; for the remark about
``We can take $\widetilde{G} = G$ ...'', see \cite[Lemma 2.4, \S 2.15]{max}.
Assume Condition (2). By Proposition \ref{PropC}, our $W$ and $G$ satisfy
Hyp(B''') in which the $G$-equivariant birational model of $W$ (also denoted as $X$ there)
and $G$ clearly satisfy Hyp(B) with $D = \delta \Delta$.
This proves Main Theorem \ref{CorA}. (See also Lemma \ref{birAct} for the birational nature of Hyp(A)).

\par \vskip 1pc
{\it Finally we prove Proposition \ref{PropE}.}
Replacing $G$ by a finite-index subgroup, we may assume that
every $G$-periodic divisor is stabilized by $G$.
As in the proof of Theorem \ref{ThB}, let
$K_X + D = P + N$
be the $\sigma$-decomposition,
where $P$ is movable and $N = \sum n_i N_i$ with $N_i$ irreducible and $n_i > 0$.
As in Claim \ref{N1}, the uniqueness of such decomposition implies that
both $P$ and $N$ are stabilized by $G$,
$P \equiv 0$, and each $N_i$ is stabilized by $G$.

Since $(X, D)$ is dlt, if we set $D_{\varepsilon} := (1 - \varepsilon) D$
then $(X, D_{\varepsilon})$ is klt for every $\varepsilon \in (0, 1]$;
see \cite[Proposition 2.41, Corollary 2.39]{KM}.
Take $\varepsilon$ small such that $\varepsilon < \min\{n_i\}$ if $N \ne 0$.
Choose $A$ as in Lemma \ref{nefbig} such that $(X, D_{\varepsilon} + A)$ is klt.
{\it We apply Proposition \ref{PropA} to the pair $(X, D_{\varepsilon} + A)$}.
So we get birational maps
$$\tau_i : X(i) \dashrightarrow X(i+1)$$ 
($0 \le i \le s$) where
$\tau_j$ ($j < s$) is either a divisorial contraction or a flip, corresponding
to a $(K_{X(j)} + D_{\varepsilon}(j))$-negative extremal ray,
and $\tau_s : X(s) \to X(s+1) = Y$ is a birational morphism.
Here we let $D(i)$, $D_{\varepsilon}(i)$ and $A(i)$ on $X(i)$ be the direct
images of $D$, $D_{\varepsilon}$ and $A$;
note that $D_{\varepsilon}(i) = (1 - \varepsilon) D(i)$.
Proposition \ref{PropE} (1) follows from Proposition \ref{PropA}.

$K_X + D \equiv P + N \equiv N$ implies that
$K_{X(s)} + D(s) \equiv N(s)$, where $N(s)$ is the direct image of $N$.
Set $\tau = \tau_{s-1} \circ \cdots \circ \tau_0 : X = X_0 \dashrightarrow X(s)$.

\begin{claim}\label{N2}
The pair $(X(s), D(s))$ has only $\Q$-factorial log canonical singularities, and
$K_{X(s)} + \Delta(s) \sim_{\Q} 0$.
\end{claim}

\begin{proof}
We prove Claim \ref{N2}.
Let $N(s) - \varepsilon D(s) = D' - D''$, with $D'$ and $D''$ effective and having no common components.
By Proposition \ref{PropA},
we have the (big and) nefness of
$$K_{X(s)} + D_{\varepsilon}(s) + A(s) \equiv N(s) - \varepsilon D(s) + A(s) = D' - D'' + A(s) ,$$
and hence the nefness of 
$$(D' - D'' + A(s))_{\, | D_i'} = (D' - D'')_{\, | D_i'}$$
for every component $D_i'$ of $D'$. Thus $D'_{\, | D_i'}$ is pseudo-effective, so
$D' = 0$, as in Claim \ref{Ns}. Now $N(s) - \varepsilon D(s) = - D'' \le 0$.
This and the choice of $\varepsilon$ then imply that the direct image $N(s)$ of $N$ is $0$.
Thus $K_{X(s)} + D(s) \equiv 0$. Hence $K_X + D = \tau^*(K_{X(s)} + D(s)) + N$.
Since the pair $(X, D)$ has only $\Q$-factorial (dlt and hence) log canonical singularities, so is $(X(s), D(s))$.
This and the known abundance theorem in the case of numerical Kodaira dimension zero,
imply that (the numerically trivial divisor) $K_{X(s)} + D(s) \sim_{\Q} 0$.
This proves Claim \ref{N2}.
\end{proof}

\par \vskip 1pc
We return to the proof of Proposition \ref{PropE}.
By Proposition \ref{PropA}, $\tau_s : X(s) \to Y$ is the
contraction of an extremal face. Since $K_{X(s)} + D(s)$ ($\sim_{\Q} 0$)
is perpendicular to every extremal ray of the face,
${\tau_s}_*(K_{X(s)} + D(s)) = K_Y + D_Y$ is an $\R$-Cartier divisor
and $K_{X(s)} + D(s) = \tau_s^*(K_Y + D_Y)$.
This and Claim \ref{N2} imply that
$K_Y + D_Y \sim_{\Q} 0$ and the pair $(Y, D_Y)$ has only log canonical singularities.
This proves Proposition \ref{PropE} (2).

By Proposition \ref{PropA}, we have the ampleness of
$K_Y + {\tau_s}_*D_{\varepsilon}(s) + A_Y = K_Y + (1 - \varepsilon) D_Y + A_Y
\sim_{\Q} -\varepsilon D_Y + A_Y$
and hence the ampleness of
$$(-\varepsilon D_Y + A_Y)_{\, | Z} = -\varepsilon {D_Y}_{\, | Z}$$
for every $G$-periodic positive-dimensional
proper subvariety $Z$ of $Y$. So $Z$ is contained in the support of $D_Y$.
This proves Proposition \ref{PropE} (3), hence the whole of Proposition \ref{PropE}.


\begin{thebibliography}{99}

\bibitem{Be}
A.~Beauville,
Some remarks on K\"ahler manifolds with $c\sb{1}=0$,
\emph{Classification of Algebraic and Analytic Manifolds}
(Katata, 1982, ed. K.~Ueno),
Progr.\ Math., \textbf{39} Birkh\"auser 1983, pp.~1--26.

\bibitem{BK}
E. \ Bedford and K. \ Kim,
Periodicities in linear fractional recurrences: Degree growth of birational surface maps,
Michigan Math. \ J. \ \textbf{54} (2006), 647-–670.

\bibitem{Bi}
C.~Birkar,
Existence of log canonical flips and a special LMMP,
Publ. \ Math. \ Inst. \ Hautes \'Etudes Sci. \textbf{115} (2012), 325 -– 368.

\bibitem{BCHM}
C.~Birkar, P.~Cascini, C.~D.~Hacon and J.~$\text{M}^{\text c}$Kernan,
Existence of minimal models for varieties of log general type,
J. \ Amer. \ Math. Soc. \ \textbf{23} (2010) 405-468.

\bibitem{CZ}
S.~Cantat and A.~Zeghib,
Holomorphic actions, Kummer examples, and Zimmer Program,
Ann. \ Sci. \ \'Ec. \ Norm. \ Sup\'er. (4) \textbf{45} (2012), no. 3, 447-–489.

\bibitem{Di}
J.~Diller,
Cremona transformations, surface automorphisms, and plane cubics,
With an appendix by Igor Dolgachev,
Michigan Math. \, J. \textbf{60} (2011), no. 2, 409–-440.

\bibitem{Di05} T.-C.~Dinh,
Suites d'applications m\'eromorphes multivalu\'ees et courants laminaires,
J. \ Geom. \ Anal. \textbf{15} (2005), no. 2, 207–-227.

\bibitem{Di12} T.-C.~Dinh,
Tits alternative for automorphism groups of compact K\"ahler manifolds, 
Acta Math. \ Vietnam. \ \textbf{37} (2012), no. 4, 513–-529.

\bibitem{DS04} T.-C.~Dinh and N.~Sibony,
Groupes commutatifs d'automorphismes d'une vari\'et\'e k\"ahlerienne compacte,
Duke Math.\ J. \textbf{123} (2004) 311--328.

\bibitem{ELMNP}
L.~Ein, R.~Lazarsfeld, M.~Mustata, M.~Nakamaye and M.~Popa,
Asymptotic invariants of base loci, Ann. Inst. Fourier (Grenoble),
\textbf{56} (2006), no. 6, 1701–-1734.

\bibitem{GKP}
D.~Greb, S.~Kebekus and T.~Peternell,
\'Etale fundamental groups of Kawamata log terminal spaces,
flat sheaves, and quotients of Abelian varieties,
arXiv:\textbf{1307.5718}

\bibitem{KH}
A.~Katok and F.~R.~Hertz,
Arithmeticity and topology of smooth actions of higher rank abelian groups,
arXiv:\textbf{1305.7262}

\bibitem{KMM}
Y.~Kawamata, K.~Matsuda and K.~Matsuki, Introduction to the minimal
model problem, in: Algebraic geometry, Sendai, 1985, 283--360, Adv.
Stud. Pure Math., Vol. \textbf{10}, 1987.

\bibitem{KeMc}
S.~Keel and J.~$\text{M}^{\text c}$Kernan, Rational curves on quasi-projective surfaces,
Mem. \ Amer. \ Math. \ Soc. \ \textbf{140} (1999), no. \textbf{669}.

\bibitem
{KM} J.~Koll\'ar and S.~Mori,
Birational geometry of algebraic varieties,
Cambridge Tracts in Math. \textbf{134},
Cambridge Univ.\ Press, 1998.

\bibitem{Li}
D.~I.~Lieberman,
Compactness of the Chow scheme: applications to automorphisms
and deformations of K\"ahler manifolds,
\emph{Fonctions de plusieurs variables complexes, III}
(\emph{S\'em.\ Fran\c{c}ois Norguet, 1975--1977}), pp.~140--186,
Lecture Notes in Math., \textbf{670}, Springer, 1978.

\bibitem{Mc}
C.~T.~McMullen,
Dynamics on blowups of the projective plane,
Publ. \ Math. \ Inst. \ Hautes \'Etudes Sci.  \ No. \textbf{105} (2007), 49 -- 89.

\bibitem{ZDA}
N.~Nakayama,
Zariski-decomposition and abundance,
MSJ Memoirs, \textbf{14},
Mathematical Society of Japan, Tokyo, 2004.

\bibitem{IntS}
N.~Nakayama,
Intersection sheaves over normal schemes, J. \ Math. \ Soc. \ Japan
\textbf{62} (2010), no. 2, 487 -– 595.

\bibitem{nz2}
N.~Nakayama and D. -Q.~Zhang,
Polarized Endomorphisms of Complex Normal Varieties,
Math. \ Ann. \ \textbf{346} (2010) 991 -- 1018.

\bibitem{Og}
K.~Oguiso,
Some aspects of explicit birational geometry inspired by complex dynamics,
ICM2014 Proceedings (to appear),
arXiv:\textbf{1404.2982}

\bibitem{SW}
N.~I.~Shepherd-Barron and P.~M.~H.~Wilson,
Singular threefolds with numerically trivial first
and second Chern classes,
J.\ Alg.\ Geom.\ \textbf{3} (1994) 265--281.

\bibitem{Z-Tits} D. -Q.~Zhang,
A theorem of Tits type for compact K\"ahler manifolds,
Invent. \ Math. \ \textbf{176} (2009) 449 -- 459.

\bibitem{JDG}
D.~-Q.~Zhang, Dynamics of automorphisms on projective complex manifolds,
J.\ Differential Geom.\ \textbf{82} (2009), no.~3, 691--722.

\bibitem{max}
D.~-Q.~Zhang,
Algebraic varieties with automorphism groups of maximal rank,
Math. \ Ann. \ \textbf{355} (2013) 131--146.

\bibitem{NullG}
D.~-Q.~Zhang,
Compact K\"ahler manifolds with automorphism groups of maximal rank,
Trans. \ Amer. \ Math. \ Soc. \textbf{366} (2014), no. 7, 3675 -- 3692.

\end{thebibliography}
\end{document}